\documentclass[a4paper,11pt]{article}
\usepackage{amsthm}
\usepackage{comment}
\usepackage{amsfonts}
\usepackage{amsmath}
\usepackage{amssymb}
\usepackage{cases}
\usepackage[mathscr]{eucal}
\usepackage{fullpage}
\usepackage{times}
\usepackage[usenames]{color}
\usepackage[dvipsnames]{xcolor}
\usepackage{hyperref}

\newcommand{\floor}[1]{\left\lfloor #1 \right\rfloor}

\newcommand{\nhalf}{\floor{\frac{n}{2}}}
\newcommand{\DES}{ \mathrm{DES}}

\newcommand{\OddDES}{ \mathrm{OddDES}}
\newcommand{\EvenDES}{ \mathrm{EvenDES}}
\newcommand{\EvenASC}{ \mathrm{EvenASC}}
\newcommand{\OddASC}{ \mathrm{OddASC}}
\newcommand{\des}{ \mathrm{des}}
\newcommand{\odes}{ \mathrm{odes}}
\newcommand{\odesb}{\mathrm{odes}_{B}}
\newcommand{\edes}{ \mathrm{edes}}
\newcommand{\edesb}{\mathrm{edes}_{B}}
\newcommand{\odesd}{\mathrm{odes}_{D}}
\newcommand{\edesd}{\mathrm{edes}_{D}}
\newcommand{\easc}{ \mathrm{easc}}
\newcommand{\oasc}{ \mathrm{oasc}}
\newcommand{\inv}{\mathrm{inv}}
\newcommand{\invb}{\mathrm{inv}_{B}}
\newcommand{\invd}{\mathrm{inv}_{D}}
\newcommand{\negsum}{\mathrm{NegSum}}
\newcommand{\qbinom}[2]{\binom{ #1 }{ #2 }_q}

\newcommand{\eascb}{ \mathrm{easc}_{B}}
\newcommand{\eascd}{ \mathrm{easc}_{D}}
\newcommand{\oascd}{ \mathrm{oasc}_{D}}

\newcommand{\sgnbinom}[2]{\mathrm{sgn}\binom{ #1 }{ #2 }}

\newcommand{\altd}{ \mathrm{altdes}}

\newcommand{\ALTD}{ \mathrm{ALTDES}}
\newcommand{\hatA}{ \mathrm{\widehat{A}}}

\newcommand{\SSS}{\mathfrak{S}}

\newcommand{\chD}{\mathcal{\widehat{D}}}
\newcommand{\cD}{\mathcal{D}}
\newcommand{\BB}{\mathfrak{B}}
\newcommand{\DD}{\mathfrak{D}}

\newcommand{\maxdrop}{\mathrm{maxdrop}}

\newcommand{\Negs}{\mathrm{Negs}}

\newcommand{\SD}{ \mathcal{SD}}

\newcommand{\Snake}{ \mathrm{Snake}}

\newcommand{\ol}[1]{\overline{#1}}

\newcommand{\Cosh}{\mathrm{Cosh}}

\newcommand{\Sinh}{\mathrm{Sinh}}

\newcommand{\Lamb}{\Lambda}

\newtheorem{theorem}{Theorem}

\newtheorem{corollary}[theorem]{Corollary}

\newtheorem{remark}[theorem]{Remark}
\newtheorem{example}[theorem]{Example}

\newtheorem{definition}[theorem]{Definition}
\newtheorem{lemma}[theorem]{Lemma}

\newcommand{\magenta}[1]{{#1}}

\newcommand{\blue}[1]{{#1}}

\title{$q$-enumeration of type B and D Eulerian polynomials 
based on parity of descents
}

\author{
	Hiranya Kishore Dey \\ 
	Department of Mathematics\\
	Indian Institute of Science, Bangalore\\
	Bangalore 560 012, India \\
	email: hiranya.dey@gmail.com
	\and 
	Umesh Shankar \\
	Department of Mathematics \\
	Indian Institute of Technology, Bombay\\
    Mumbai 400 076, India\\
	email: 204093001@iitb.ac.in
	\and 
		Sivaramakrishnan Sivasubramanian \\
		Department of Mathematics\\
		Indian Institute of Technology, Bombay\\
		Mumbai 400 076, India\\
		email: krishnan@math.iitb.ac.in
		\and 
}

\begin{document}

\maketitle

\begin{center}
{\bf \large Abstract} 
\end{center} 
Carlitz and Scoville in 1973 considered a four variable polynomial 
that enumerates permutations in $\SSS_n$ with respect to the parity 
of its descents and ascents. In recent work,  Pan and Zeng  proved a $q$-analogue 
of Carlitz-Scoville's generating function by enumerating permutations with 
the above four statistics along with the inversion number. 
Further, they also proved a type B analogue by enumerating signed 
permutations with respect to the parity of descents and ascents. 
In this work we prove a $q$-analogue of the type B result of Pan and Zeng 
by enumerating permutations in $\BB_n$ with the above four statistics and
the type B inversion number. We also obtain a $q$-analogue 
of the generating function for the type B bivariate alternating 
descent polynomials.  We consider a similar five-variable polynomial 
in the type D Coxeter groups as well and give their egf. Alternating descents 
for the type D groups were 
previously also defined by Remmel, but our definition is slightly different. 
As a by-product of our proofs, we get bivariate $q$-analogues 
of Hyatt's recurrences for the type B and type D Eulerian polynomials. 
	Further corollaries of our results are some symmetry relations 
for these polynomials and $q$-analogues of generating functions for
snakes of types B and D.

\section{Introduction}
\label{sec:intro}

For a positive integer $n$, let $[n] = \{1,2, \ldots, n\}$ and let $\SSS_n$ 
be the set of permutations of $[n]$. For a permutation 
$\pi= \pi_1, \pi_2, \dots, \pi_n \in \SSS_n$, an index $i \in [n-1]$ is said to be a 
descent of $\pi$ if $\pi_i > \pi_{i+1}.$ 
Define $\DES(\pi)= \{i  \in [n-1]: \pi_i> \pi_{i+1} \}$ to be the set 
of descents of $\pi$ and let $\des(\pi) =| \DES(\pi)|.$ The classical 
Eulerian polynomial is defined as the generating function of the 
descent statistic over $\SSS_n,$ that is,  
$$A_n(t)=\sum_{\pi \in \SSS_n} t^{\des(\pi)}.$$

These polynomials are very well-studied.   The books
by Foata and Schutzenberger 
\cite{foata-schutzenberger-eulerian} and 
by Petersen \cite{petersen-eulerian-nos-book}
contain many interesting results on these polynomials. 
An index $i \in [n]$ is called an ascent of $\pi \in \SSS_n$ 
if $\pi_i < \pi_{i+1}$.  Taking parity of the position of the
descents, one can define {\sl odd ascents, odd descents, even
ascents} and {\sl even descents}.   Formally, let 
$\EvenDES(\pi) = \{i \in  [n-1]: \pi_i > \pi_{i+1}, i \mbox{ is even}\}$,
$\EvenASC(\pi) = \{i \in  [n-1]: \pi_i < \pi_{i+1}, i \mbox{ is even}\}$,
$\OddDES(\pi) = \{i \in  [n-1]: \pi_i > \pi_{i+1}, i \mbox{ is odd}\}$ and
$\OddASC(\pi) = \{i \in  [n-1]: \pi_i < \pi_{i+1}, i \mbox{ is odd}\}$.
Carlitz and Scoville 
in \cite{carlitz-scoville-perm-rises-falls}
considered the polynomial

\begin{equation}
\label{eqn:quad_stat}
A_n(s_0,s_1,t_0,t_1) = \sum_{\pi \in \SSS_n} s_0^{\easc(\pi)} 
s_1^{\oasc(\pi)} t_0^{\edes(\pi)} t_1^{\odes(\pi)}.
\end{equation}

They gave the exponential generating function (egf henceforth) for the 
above polynomial (see Theorem \cite[Theorem 3.1]{carlitz-scoville-perm-rises-falls}).
Pan and Zeng considered a $q$-analogue of the above polynomial by 
adding the inversion number as well.  For $\pi \in \SSS_n$ define 
$\inv(\pi) = |\{1 \leq i < j \leq n: \pi_i > \pi_j \}|$.  
They considered

\begin{equation}
\label{eqn:penta_stat}
A_n(s_0,s_1,t_0,t_1,q) = \sum_{\pi \in \SSS_n} s_0^{\easc(\pi)} 
s_1^{\oasc(\pi)} t_0^{\edes(\pi)} t_1^{\odes(\pi)} q^{\inv(\pi)}.
\end{equation}

Pan and Zeng gave the following egf for $A_n(s_0,s_1,t_0,t_1,q).$  
For integers $i\ge 0$, define  $[i]_q=(1+q+\cdots+q^{i-1})$ and define
$n!_q = \prod_{i=1}^n[i]_q$.   
Recall that $\displaystyle \mathrm{e}_q(u)=\sum_{n\geq 0} 
\frac{u^n}{[n]_q!}$. Separating the odd and even terms,  let 
$$\cosh_q(u)=\frac{\mathrm{e}_q(u)+\mathrm{e}_q(-u)}{2} 
\hspace{5 mm}  \mbox{ and } \hspace{5 mm} 
\sinh_q(u)=\frac{\mathrm{e}_q(u)-\mathrm{e}_q(-u)}{2}.$$
Pan and Zeng in \cite[Theorem 1.2]{pan_zeng-enum_perm_parity_descent}
showed the following (they use the variables 
$x,y$ for what we denote $t,s$ respectively.)

\begin{theorem}[Pan and Zeng]
\label{thm:carlitz_scoville}
Let 
$\alpha = \sqrt{(t_0-s_0)(t_1-s_1)}$.  Then, 
\begin{eqnarray}
\label{eqn:carlitz_scoville}
\sum_{n \geq 1} A_n(s_0,s_1,t_0,t_1,q) u^n/n!_q \hspace{8 cm}  \nonumber \\
= \frac{(s_1 + t_1)\cosh_q(\alpha u) + \alpha \sinh_q(\alpha u) -t_1
(\cosh_q^2(\alpha u) - \sinh_q^2(\alpha u)) - s_1}
{s_0s_1 - (s_0t_1 + s_1t_0)\cosh_q(\alpha u) + t_0t_1
(\cosh_q^2(\alpha u) - \sinh_q^2(\alpha u))  }.
\end{eqnarray}
\end{theorem}

Using the same notation, Theorem \ref{thm:carlitz_scoville}
gives rise to an identity for the bivariate Eulerian
polynomial and the bivariate alternating Eulerian polynomial.
It is easy to see (and noted by Pan and Zeng
\cite{pan_zeng-enum_perm_parity_descent}) among the four statistics that
involve descents and ascents in Theorem \ref{thm:carlitz_scoville}, there are 
choices of {\sl two} of them which determine the other two statistics.  
Indeed, using our result, we get type B and type D counterparts of 
Theorem \ref{thm:carlitz_scoville}.  \blue{These are presented 
as Theorem \ref{q-egf-type-b-counterpart-pan-zeng} and Theorem 
\ref{thm:pan_zeng_four_var_type-d-version} respectively.
}

Pan and Zeng in  \cite{pan_zeng-enum_perm_parity_descent} also 
gave a type B counterpart of 
these identities without the variable $q$ (that is, 
without taking type $B$ inversions into account).  For a positive 
integer $n$, let $[\pm n] = \{\pm 1, \pm 2, \ldots, \pm n\}$.
$\BB_n$ is the set of permutations $\pi$ of $[\pm n]$
that satisfy $\pi(-i) = -\pi(i)$.  Let $\pi_0 = 0$ for all
$\pi \in \BB_n$ and let $[n]_0  = \{0,1,2,\ldots,n\}$.  Define
$\EvenDES_B(\pi) = \{i \in  [n-1]_0: \pi_i > \pi_{i+1}, i \mbox{ is even}\}$,
$\EvenASC_B(\pi) = \{i \in  [n-1]_0: \pi_i < \pi_{i+1}, i \mbox{ is even}\}$,
$\OddDES_B(\pi) = \{i \in  [n-1]_0: \pi_i > \pi_{i+1}, i \mbox{ is odd}\}$ and
$\OddASC_B(\pi) = \{i \in  [n-1]_0: \pi_i < \pi_{i+1}, i \mbox{ is odd}\}$.
Define $\odes_B(\pi) = |\OddDES_B(\pi)|$, $\edes_B(\pi) = |\EvenDES_B(\pi)|$,
$\oasc_B(\pi) = |\OddASC_B(\pi)|$ and lastly $\easc_B(\pi) = |\EvenASC_B(\pi)|$.
Further, define
\begin{equation} 
\label{eqn:defn_biv_eul_alt_eul}
B_n(s,t) = \sum_{\pi \in \BB_n} s^{\edes_B(\pi)} t^{\odes_B(\pi)}
\hspace{5 mm} \mbox{ and } \hspace{5 mm}
\hat{B}_n(s,t) = \sum_{\pi \in \BB_n} s^{\easc_B(\pi)} t^{\odes_B(\pi)}.
\end{equation}
Setting $s=t$ in the polynomial $\hat{B}_n(s,t)$ gives $\hat{B}_n(t)$, 
the type B alternating Eulerian
polynomial which has been studied for example by 
Ma, Fang, Mansour and Yeh 
\cite{alt-euler-poly-left-peak-poly_ma_fang_mans_yeh}.

\begin{theorem}[Pan and Zeng]
\label{eqn:from_carlitz_scoville_typeb_eulerian}
Let $\alpha = (1-s)(1-t)$.  Then, we have 
\begin{eqnarray}
\sum_{n \geq 1} B_{2n}(s,t)\frac{u^{2n}}{(2n)!} & = & 
\frac{(s+t)\sum_{n \geq 0} \frac{\alpha^{n} (2u)^{2n} }{(2n)!} + \sum_{n \geq 0} \frac{\alpha^{n+1} u^{2n}}{(2n)!} -(1+st) }
{(1 +st) - (s+t)\sum_{n \geq 0} \frac{\alpha^n (2u)^{2n}}{(2n)!}  }
\label{eqn:egf_biv_eulerian_typeB_even}, \\
\sum_{n \geq 0} B_{2n+1}(s,t)\frac{u^{2n+1}}{(2n+1)!} & = & 
\frac{(s^2-1)(t-1) \sum_{n \geq 0} \frac{\alpha^{n} u^{2n+1} }{(2n+1)!}  }
{(1 + st) - (s+ t)\sum_{n \geq 0} \frac{\alpha^{n} (2u)^{2n}}{(2n)!}  }.
\label{eqn:egf_biv_eulerian_typeB_odd} 
\end{eqnarray}
\end{theorem}

They also gave similar results about the type B alternating 
descent polynomials.  Their result is as follows.

\begin{theorem}[Pan and Zeng]
\label{eqn:from_carlitz_scoville_typeb_alt_eulerian}
Let $\alpha = (1-s)(1-t)$.  Then, we have 
\begin{eqnarray}
\sum_{n \geq 1} \hat{B}_{2n}(s,t)u^{2n}/(2n)! & = & 
\frac{(1+st) \sum_{n \geq 0} \frac{(-\alpha)^{n} (2u)^{2n} }{(2n)!} + \sum_{n \geq 0} \frac{(-\alpha)^{n+1} u^{2n}}{(2n)!} -(s+t) }
{(s +t) - (1+st)\sum_{n \geq 0} \frac{(-\alpha)^n (2u)^{2n}}{(2n)!}  }
\label{eqn:egf_biv_alt_eulerian_typeB_even}, \\
\sum_{n \geq 0} \hat{B}_{2n+1}(s,t)u^{2n+1}/(2n+1)! & = & 
\frac{(1+s) \sum_{n \geq 0} \frac{(-\alpha)^{n+1} u^{2n+1} }{(2n+1)!}  }
{(s + t) - (1+st)\sum_{n \geq 0} \frac{\alpha^{n} (2u)^{2n}}{(2n)!}  }
\label{eqn:egf_biv_alt_eulerian_typeB_odd}.
\end{eqnarray}
\end{theorem}

$$
\mbox{Let } H_0(s,t,u) = \sum_{n \geq 0} B_{2n}(s,t) \frac{u^{2n}}{(2n)!} 
\hspace{5 mm}  \mbox{ and } \hspace{5 mm} H_1(s,t,u) = \sum_{n \geq 0} 
B_{2n+1}(s,t)\frac{u^{2n+1}}{(2n+1)!}.$$
Recall that $\cosh(x) = \frac{1}{2}\Big( \exp(x) + \exp(-x) \Big) $ and 
$\sinh(x) = \frac{1}{2} \Big(\exp(x) - \exp(-x) \Big)$.  
\begin{equation}
\label{eqn:defn_M}
\mbox{Define } M^2 = \alpha.
\end{equation}  
It is easy to see that the following alternate form can be used to 
state Theorem \ref{eqn:from_carlitz_scoville_typeb_eulerian}.  

\begin{theorem}[Pan and Zeng]
\label{thm:pan_zeng_biv_eul_egf}
With the above notation, 
\begin{eqnarray}
\label{eqn:typeb-eventerms-pz}
    H_0(s,t,u) & = & \frac{M^2\cosh(uM)}{M^2\cosh^2(uM)-(s+1)(t+1)\sinh^2(uM)}, \\
    \label{eqn: typeb-oddterms-pz}
    H_1(s,t,u) & = & \frac{M(s+1)\sinh(uM)}{M^2\cosh^2(uM)-(s+1)(t+1)\sinh^2(uM)}.
\end{eqnarray}
\end{theorem}

\noindent 
Recall that length in Type B Coxeter groups is defined as follows 
(see \cite[Page 294]{petersen-eulerian-nos-book}).  For $\pi \in \BB_n$,
\begin{equation}
\label{eqn:typeb-inv-defn}
\inv_B(\pi)=  | \{ 1 \leq i < j \leq n : \pi_i > \pi_j \} |+  
| \{ 1 \leq i < j \leq n : -\pi_i > \pi_j \}| +|\Negs(\pi)|,
\end{equation} 
where $\Negs(\pi) = \{\pi_i : i > 0, \pi_i < 0 \}$.
Further, recall the definition of $\odes_B(\pi)$ and 
$\edes_B(\pi)$ from earlier.  Define
\begin{eqnarray}
\label{eqn:defn_eulerian_bi_tri_variant}
B_n(t,q) = \sum_{\pi \in \BB_n} t^{\des_B(\pi)}q^{\inv_B(\pi)}
\hspace{-2 mm} 
& \mbox{and }& 
\hspace{-2 mm} 
B_n(s,t,q) = \sum_{\pi \in \BB_n} s^{\edes_B(\pi)}t^{\odes_B(\pi)}
q^{\inv_B(\pi)},\\
\label{eqn:type-b-defns-odd-even-biv-eulerian}
H_0(s,t,q,u)=\sum_{k \geq 0} B_{2n}(s,t,q)\frac{u^{2n}}{B_{2n}(1,q)} 
\hspace{-3 mm}
& \mbox{ and } & 
\hspace{-4 mm}
H_1(s,t,q,u)=\sum_{k \geq 0} B_{2n+1}(s,t,q)\frac{u^{2n+1}}{B_{2n+1}(1,q)}.
\end{eqnarray}

Let $\displaystyle \exp_B(u;q)=\sum_{n \geq 0}\frac{u^n}{B_n(1,q)}$.
As before, we separate terms with odd and even exponents and define
$$\cosh_B(u;q)=\frac{\exp_B(u;q)+\exp_B(-u;q)}{2} 
\hspace{5 mm}  \mbox{ and } \hspace{5 mm} 
\sinh_B(u;q)=\frac{\exp_B(u;q)-\exp_B(-u;q)}{2}.$$
With this notation, our first main result is the 
following $q$-analogue of Theorem \ref{thm:pan_zeng_biv_eul_egf}.

\begin{theorem}
\label{thm:from_carlitz_scoville_typeb_eulerian_q-anal}
We have
\noindent
\begin{eqnarray}
H_0(s,t,q,u) & = & \frac{(1-s)\bigg(\big(1-t\cosh_q(Mu)\big){\cosh_B(Mu;q)}+t\sinh_q(Mu){\sinh_B(Mu;q)}\bigg)}{1-(s+t)\cosh_q(Mu)+st\mathrm{e}_q(Mu)\mathrm{e}_q(-Mu)}, 
\label{eqn:from_carlitz_scoville_typeb_eulerian_q-anal_even}\\
H_1(s,t,q,u) & = & \frac{M\bigg(\big(1-s\cosh_q(Mu)\big){\sinh_B(Mu;q)}+s\sinh_q(Mu){\cosh_B(Mu;q)}\bigg)}{1-(s+t)\cosh_q(Mu)+st\mathrm{e}_q(Mu)\mathrm{e}_q(-Mu)}. 
\label{eqn:from_carlitz_scoville_typeb_eulerian_q-anal_odd}
\end{eqnarray}
\end{theorem}

Theorem \ref{thm:from_carlitz_scoville_typeb_eulerian_q-anal} is
proved in Subsection \ref{subsec:type_b_gen_fns}.
Recalling \eqref{eqn:defn_biv_eul_alt_eul}, define
$$\hat{H}_0(s,t,u) = \sum_{n \geq 0} \hat{B}_{2n}(s,t) \frac{u^{2n}}{(2n)!} 
\hspace{5 mm}  \mbox{ and } \hspace{5 mm} \hat{H}_1(s,t,u) = \sum_{n \geq 0} 
\hat{B}_{2n+1}(s,t)\frac{u^{2n+1}}{(2n+1)!}.$$

	We have rewritten Theorem
\ref{eqn:from_carlitz_scoville_typeb_eulerian}
as Theorem \ref{thm:pan_zeng_biv_eul_egf} and stated our 
generalization as Theorem 
\ref{thm:from_carlitz_scoville_typeb_eulerian_q-anal}. 
Similarly, it is easy to see that Theorem 
\ref{eqn:from_carlitz_scoville_typeb_alt_eulerian} can 
be rewritten as follows.  

\begin{theorem}[Pan and Zeng]
\label{thm:pan_zeng_biv_alt_egf}
With the above notation, 
\begin{eqnarray}
\label{eqn:alttypeb-eventerms-pz}
    \hat{H}_0(s,t,u) & = & \frac{-(s-1)(t-1)\cos(Mu)}{s+t-(ts+1)\cos(2Mu)}, \\
    \label{eqn:alttypeb-oddterms-pz}
    \hat{H}_1(s,t,u) & = & \frac{-M(s+1)\sin(Mu)}{s+t-(ts+1)\cos(2Mu)}.
\end{eqnarray}
\end{theorem}
Define $\hat{B}_{n}(s,t,q)=
\sum_{\pi \in \BB_{n}} t^{\odes_B(\pi)} s^{\easc_B(\pi)}q^{\inv_B(\pi)}$  
and let 
$$ \hat{H}_1(s,t,q,u)=
\sum_{n\geq 0} \hat{B}_{2n+1}(s,t,q)\frac{u^{2n+1}}{B_{2n+1}(1,q)},
\hspace{3 mm}  \mbox{ and } \hspace{3 mm} 
\widehat{H}_0(s,t,q,u) = 
\sum_{n\geq 0} \hat{B}_{2n}(s,t,q)\frac{u^{2n}}{B_{2n}(1,q)}.$$

Moreover, let 
$$\cos_B(u;q)=\frac{\exp_B(iu;q)+\exp_B(-iu;q)}{2} \hspace{5 mm}  \mbox{ and } 
\hspace{5 mm} \sin_B(u;q)=\frac{\exp_B(iu;q)-\exp_B(-iu;q)}{2}.$$
Another of our main results is the following $q$-analogue of Theorem 
\ref{thm:pan_zeng_biv_alt_egf}. 

\begin{theorem}
\label{thm:from_carlitz_scoville_typeb_alt_eulerian_q-anal}
We have 
\begin{eqnarray}
\widehat{H}_0(s,t,q,u) & = & \frac{(s-1)\bigg( (1-t\cos_q(Mu))\cos_{B}(Mu;q)-t\sin_q(Mu)\sin_{B}(Mu;q)\bigg)}{s+t\mathrm{e}_q(iMu)\mathrm{e}_q(-iMu)-(ts+1)\cos_q(Mu)},  
\label{eqn:from_carlitz_scoville_typeb_alt_eulerian_q-anal_even} \\
\widehat{H}_1(s,t,q,u) & = & \frac{-M\bigg((s-\cos_q(Mu)\sin_{B}(Mu;q)+\sin_q(Mu)\cos_{B}(Mu;q) \bigg)}{s+t\mathrm{e}_q(iMu)\mathrm{e}_q(-iMu)-(ts+1)\cos_q(Mu)}.
\label{eqn:from_carlitz_scoville_typeb_alt_eulerian_q-anal_odd} 
\end{eqnarray}
\end{theorem}


The proof of Theorem 
\ref{thm:from_carlitz_scoville_typeb_alt_eulerian_q-anal} 
is also given in Subsection \ref{subsec:type_b_gen_fns}.
We move to our counterpart of this result to type D Coxeter groups 
$\DD_n$.    Recall that $\DD_n$ is the subgroup of $\BB_n$ consisting of 
the signed permutations which have an even number of negative 
signs.  We denote $-1$ as $\overline{1}$ and 
for $\pi = \pi_1, \pi_2, \ldots, \pi_n \in \DD_n$, define 
$\pi_{\overline{1}}= -\pi_1$ and let 
$\DES_D(\pi) = \{i \in \lbrace -1,1,\dots,n-1\rbrace : 
\pi_i > \pi_{|i|+1}  \}$ be its set of descents.  Let 
$\des_D(\pi) = |\DES_D(\pi)|$. Moreover, let 
$\OddDES_D(\pi) = \{i \in [-1,n-1]-\{0\}: \pi_i > \pi_{|i|+1} \mbox{ and } 
i \mbox{ is odd}\}$ be the set of odd indices where descents occur in 
$\pi$ and similarly let 
$\EvenDES_D(\pi) = \{i \in [-1,n-1]-\{0\}: \pi_i > \pi_{i+1} \mbox{ and } 
i \mbox{ is even}\}$.  Let $\odes_D(\pi) = |\OddDES_D(\pi)|$ and 
$\edes_D(\pi) = |\EvenDES_D(\pi)|$.
Recall that length in Type D Coxeter groups $\inv_D$ is defined as 
follows (see \cite[Page 302]{petersen-eulerian-nos-book}).
\begin{equation}
\label{eqn:typed-inv-defn}
\inv_D(\pi)=  | \{ 1 \leq i < j \leq n : \pi_i > \pi_j \} |+  
| \{ 1 \leq i < j \leq n : -\pi_i > \pi_j \}|.  
\end{equation}

Remmel in \cite{Remmelaltdes} has given a definition of alternating
descent for type D Coxeter groups based on a total order on the 
elements of $[\pm n]$.  The polynomial that Remmel gets is different
from the one we have. Remmel's main result is a joint distribution 
of alternating descents and alternating major index in type B and D 
Coxeter groups.  Below, we consider a slightly different polynomial
enumerating  alternating descents and type D inversion number 
in $\DD_n$.  Our definition uses the parity of the position of 
descents as before.
Formally, define 
$$D_n(t,q) = \sum_{\pi \in \DD_n} t^{\des_D(\pi)}q^{\inv_D(\pi)}
\mbox{ and }
D_n(s,t,q) = \sum_{\pi \in \DD_n} s^{\edes_D(\pi)}t^{\odes_D(\pi)}
q^{\inv_D(\pi)}.  \mbox{  Define}$$
$$\mathcal{D}_0(s,t,q,u)=\sum_{k > 0} D_{2k}(s,t,q)\frac{u^{2k}}{D_{2k}(1,q)},
\mathcal{D}_1=
\sum_{k > 0} D_{2k+1}(s,t,q)\frac{u^{2k+1}}{D_{2k+1}(1,q)}.$$

\magenta{Define $\hat{D}_{n}(s,t,q)=
\sum_{\pi \in \DD_{n}} t^{\odes_D(\pi)} s^{\easc_D(\pi)}q^{\inv_D(\pi)}$  
and let 
\begin{equation}
\label{eqn:defn_four_var_egf_type-D}
\chD_0(s,t,q,u) = 
\sum_{n\geq 1} \hat{D}_{2n}(s,t,q)\frac{u^{2n}}{D_{2n}(1,q)},
\hspace{1 mm}   
\chD_1(s,t,q,u)= 
\sum_{n\geq 1} \hat{D}_{2n+1}(s,t,q)\frac{u^{2n+1}}{D_{2n+1}(1,q)}.
\end{equation}
}

Moreover, let
$\displaystyle \exp_D(u;q)=\sum_{n\geq 0} \frac{u^n}{D_n(1,q)}.$
We split its odd and even parts and write
$$\cosh_D(u;q)=\frac{\exp_D(u;q)+\exp_D(-u;q)}{2} 
\hspace{5 mm}  \mbox{ and } \hspace{5 mm} 
\sinh_D(u;q)=\frac{\exp_D(u;q)-\exp_D(-u;q)}{2}.$$

Recalling $M$ from \eqref{eqn:defn_M}, let
\begin{eqnarray*}
\mathrm{OD}& = & ut^2(\cosh_q(Mu)-1) +\frac{(1-t)M}{(1-s)}
(\sinh_{D}(Mu;q)-Mu)+\frac{2t(1-t)}{M}(\sinh_q(Mu)-Mu), \\ 
\mathrm{ED}& = & 2t(\cosh_q(Mu)-1)
+(1-t)(\cosh_{D}(Mu;q)-1)+\frac{ut^2(1-s)}{M}\sinh_q(Mu).
\end{eqnarray*}
For type D Coxeter groups, our main results are 
the following.
\begin{theorem}
\label{thm:carlitz_scoville_typed_eulerian_q-anal}
We have the egfs
\begin{eqnarray}
\mathcal{D}_0(s,t,q,u)=\frac{\mathrm{ED}(1-t\cosh_q(Mu))+\mathrm{OD}(\frac{t(1-s)}{M}\sinh_q(Mu))}{1-(s+t)\cosh_q(Mu)+st\mathrm{e}_q(Mu)\mathrm{e}_q(-Mu)}, \\ 
\label{eqn:from_carlitz_scoville_typed_alteulerian_q-anal_even}
 \mathcal{D}_1(s,t,q,u)=\frac{\mathrm{OD}(1-s\cosh_q(Mu))+\mathrm{ED}(\frac{s(1-t)}{M}\sinh_q(Mu))}{1-(s+t)\cosh_q(Mu)+st\mathrm{e}_q(Mu)\mathrm{e}_q(-Mu)}. 
\label{eqn:from_carlitz_scoville_typeb_alteulerian_q-anal_odd}
\end{eqnarray}
\end{theorem}

\blue{
\begin{theorem}
\label{thm:type-d-alt-eul}
We have the egfs
 \begin{eqnarray}
\chD_0(s,t,q,u)
& =& \frac{T'(\mathrm{ED})(1-t\cos_q(Mu))-T'(\mathrm{OD})
(\frac{t(s-1)\sqrt{s}}{Ms}\sin_q(Mu))}{s-(st+1)
\cos_q(Mu)+t\mathrm{e}_q(iMu)\mathrm{e}_q(-iMu)},
\\
\chD_1(s,t,q,u)
& = & \frac{\frac{T'(\mathrm{OD})}{\sqrt{s}}(s-\cos_q(Mu))-T'(\mathrm{ED})
(\frac{(1-t)}{M}\sin_q(Mu))}{s-(st+1)\cos_q(Mu)+t\mathrm{e}_q(iMu)\mathrm{e}_q(-iMu)}.
\end{eqnarray}
where
\begin{eqnarray*}
T'(\mathrm{OD})& =& \sqrt{s}ut^2(\cos_q(Mu)-1)-\frac{(1-t)M}{(s-1)\sqrt{s}}
(\sin_{D}(Mu;q)-Mu)\\
& & +\frac{2t(1-t)\sqrt{s}}{M}(\sin_q(Mu)-Mu), \\
T'(\mathrm{ED})& =& 2t(\cos_q(Mu)-1)+\frac{ut^2(s-1)\sqrt{s}}{sM}\sin_q(Mu)
+\frac{(1-t)}{t}(\cosh_{D}(Mu;q)-1).
\end{eqnarray*}
\end{theorem}
}

The proof of Theorem
\ref{thm:carlitz_scoville_typed_eulerian_q-anal} and Theorem
\ref{thm:type-d-alt-eul}
appear in Subsection \ref{subsec:type_d_gen_fns}.
It can be checked that Theorem 
\ref{thm:carlitz_scoville_typed_eulerian_q-anal}
refines a result of Reiner
\cite[Corollary 4.5]{reiner-descents-weyl}
for type D Euler-Mahonian polynomials.
\magenta{Our proofs in both the type B and type D cases use an 
inclusion-exclusion based argument.}


\subsection{Refining Hyatts recurrences for the Type B and Type 
D Eulerian polynomial}
As an outcome of our proofs, we get a refinement of Hyatt's 
recurrence for the type B and type D Eulerian polynomials.  
Hyatt in \cite{hyatt-recurrences_eulerian_typeBD} gave
the following recurrences for Eulerian polynomials of 
types B.   We partition $\BB_n$ based on the sign of
the last element.  Define $\BB_n^+ = \{\pi \in \BB_n:
\pi_n > 0 \}$ contain the elements of $\BB_n$ with last
element being positive and let $\BB_n^- = \BB_n - \BB_n^+$.
Define 
$B_n^+(t) = \sum_{\pi \in \BB_n^+} t^{\des_B(\pi)}$.
The following result is due to Hyatt.

\begin{theorem}[Hyatt]
\label{thm:hyatt_recurrences}
For integers $n \geq 1$, we have
\begin{eqnarray*}
B_n^+(t) & = & \sum_{k=0}^{n-1} \binom{n}{k} B_k(t) (t-1)^{n-k-1}.
\end{eqnarray*}
\end{theorem}

Our extension of Theorem \ref{thm:hyatt_recurrences} involves the 
following polynomial.    Define
\begin{equation}
\label{defn:poly_pm_like_hyatt}
B_n^{\pm}(s,t,q) = \sum_{\pi \in \BB_n^{\pm}}s^{\edes_B(\pi)} t^{\odes_B(\pi)}
q^{\invb(\pi)} .
\end{equation}
Our type B generalization is the following.

\begin{theorem}
\label{thm:typeb_plus}
For even positive integers $n$, we have
\begin{eqnarray}
\label{eqn:typeb_plus_biv_even}
B_n^+(s,t,q) & = & 
\sum_{r=0}^{\frac{n}{2}-1} q^{\binom{2r+1}{2}}\qbinom{n}{2r+1} 
B_{n-2r-1}(s,t,q)(s-1)^r (t-1)^r \nonumber \\
\label{eqn:typeb_hyatt_multivariate_even}
  & & + \sum_{r=1}^{\frac{n}{2}} q^{\binom{2r}{2}}\qbinom{n}{2r} B_{n-2r}(s,t,q)(s-1)^{r-1} (t-1)^r.
\end{eqnarray}
For odd positive integers $n$, we have
\begin{eqnarray}
\label{eqn:typeb_plus_biv_odd}
B_n^+(s,t,q) & = & 
\sum_{r=0}^{\lfloor \frac{n}{2}\rfloor} q^{\binom{2r+1}{2}}\qbinom{n}{2r+1} 
B_{n-2r-1}(s,t,q)(s-1)^r (t-1)^r \nonumber \\ 	
\label{eqn:typeb_hyatt_multivariate_odd}
& & + \sum_{r=1}^{\lfloor \frac{n}{2} \rfloor} q^{\binom{2r}{2}}\qbinom{n}{2r} B_{n-2r}(s,t,q)(s-1)^{r} (t-1)^{r-1}.
\end{eqnarray}
\end{theorem}

It is clear that setting $q=1$ and $s=t$ in Theorem \ref{thm:typeb_plus} 
gives us Theorem \ref{thm:hyatt_recurrences}. 
The proof of Theorem \ref{thm:typeb_plus} 
appears in Subsection \ref{subsec:hyatt_type_b_q_anal}. 
For Type D Coxeter groups, our analogous result 
is Theorem \ref{thm:typed_plus}. 

\subsection{More consequences}
\magenta{Another outcome of our results are some symmetry relations.
For the type B case, our results are Theorem 
\ref{thm:typeb_minus}
and Lemma \ref{thm: type_b_reciprocal}.  For the type D case, 
our symmetry results are Theorem \ref{thm:typed_minus}
and Corollary \ref{thm: type_d_reciprocal}.
}

From the $q$-analogue of our generating function, we 
naturally get a $q$-analogue of the enumeration of type B and 
type D snakes.  These results are presented in Section
\ref{sec:snakes}.  Enumeration of type B and D snakes with respect
to some statistics and thus $q$-analogues have been obtained, 
see for example, Verges \cite{verges-enum-snakes-cycle-alt}.  
However, to the best of our knowledge, we have not seen $q$-analogues 
involving the appropriate length function in these groups.

\section{Type B results}
\label{sec:type-b-results}
Recall that $\BB_n$ is the set of permutations of 
$[\pm n] = \{\pm 1, \pm 2,\ldots,\pm n\}$ satisfying $\pi(-i)= - \pi(i)$. 
We think of $\pi$ as a word $\pi = \pi_0, \pi_1, \pi_2, \ldots, \pi_n$ where 
$\pi_i = \pi(i)$ and $\pi_0=0.$  



For positive integers $n$ and an integer $i$ with $0 \leq i \leq n$, 
let  $\binom{[n]}{i}= \{ A \subseteq [n]: |A|=i \}$ be 
the set of subsets of $[n]$ with cardinality $i$.
We define a signed subset $(A,\epsilon)$ to be a 
subset $A\subseteq [n]$ and $\epsilon$ is a string of 
signs $\pm$ of length $|A|$.  Here, each element $a_i \in A$ has either a
positive or a negative sign, encoded by $\epsilon_i$, attached to it. 
When $a \in A$, we denote a positive signed $a$ just by $a$ and a negative 
signed $a$ by $\overline{a}$. The set of all signed subsets of 
size $i$ of $[n]$ will be denoted as $\sgnbinom{[n]}{i}$. 
Clearly, $|\sgnbinom{[n]}{i}| = 2^i\binom{n}{i}$.

Let $G_{n,i}$ be the set of signed permutations $\pi \in \BB_n$ 
such that the last $n-i$ elements of $\pi$ are increasing, 
that is we have 
$\pi_{i+1} < \pi_{i+2} < \dots <\pi_{n-1} < \pi_n.$ 
It is easy to see that $|G_{n,i}|= 2^n\binom{n}{i}i!.$
Define $G_{n,-1}$ to be the signed permutation $\pi=0,1,2, \cdots, n.$, the signed permutation whose $n+1$ elements are increasing.

Let $\sigma=0, \sigma_1, \cdots, \sigma_{n-i} \in \BB_{n-i}$ 
and $(A,\epsilon) \in \sgnbinom{[n]}{i}$ 
be a signed subset. 
Moreover, let $[n]-A=\{c_1, c_2, \dots, c_{n-i}\}$ be written 
in ascending order, 
that is with $c_1 < c_2 < \dots < c_{n-i}.$  We define a map 
$h: \BB_{n-i} \to \BB_{ \{c_1, c_2, \dots, c_{n-i} \} }$ 
which for $1 \leq k \leq n-i$, maps $k$ to $c_k$ and 
preserves the sign.  Formally,

\begin{equation}
\label{eqn:defn_h}
h(\sigma)=0, \pi_1, \pi_2, \dots, \pi_{n-i},
\end{equation}
where for $1 \leq i \leq n-i$, if $|\sigma_i|=k$ then  
$|\pi_i|=c_k$ and $\pi_i$ has the same sign as $\sigma_i$. 
This map $h$ is clearly a bijection and is hence invertible.

By inverting the map $h$ on the elements of $[0,n]-A$ 
and appending the elements of $(A,\epsilon)$ in ascending order, 
we get a signed permutation in $G_{n,n-i}$. This map is also invertible, 
and thus we have a bijection $f: \BB_{n-i}\times \sgnbinom{[n]}{i} 
\mapsto G_{n,n-i}$ defined below.
Let $\sigma \in \BB_{n-i}$ and $(A,\epsilon) \in \sgnbinom{[n]}{i}$. 
For a set $S$ (resp. a signed set $(S,\epsilon)$), by $[S]$ (respectively 
by $[(S,\epsilon)]$), we denote the string obtained by writing 
the elements of $S$ (respectively $(S,\epsilon)$)  in ascending 
order in the usual linear order of $\mathbb{Z}$. 
Define $f(\sigma,(A,\epsilon))=h(\sigma)[(A,\epsilon)]$ where 
$h(\sigma)[(A,\epsilon)]$ denotes the juxtaposition of $h(\sigma)$ and 
$[(A,\epsilon)]$. 

\begin{example}
	Let $n=7, i=4, \sigma= 0,\overline{2},1,3 \in \BB_3$ and 
	$(A,\epsilon)=\{1,\overline{4},5,\overline{6} \}$ be a
	signed subset of $\sgnbinom{[7]}{4}$. Then, 
	$[0,n]-A = \{0,2,3,7\}$ and thus $h(\sigma)  = 0 \overline{3}27.$ 
	Moreover, we have $[[0,n]-A]=0,2,3,7$ and 
	$[(A,\epsilon)]= \overline{6}, \overline{4}, 1, 5$. Therefore, 
	$f(\sigma, (A,\epsilon))= 0,\overline{3},2,7,\overline{6},\overline{4},1,5.$
	We also have $f([[0,7]-A], (A,\epsilon)) =0,2,3,7,\ol{6}, \ol{4},1,5.$ 
\end{example}

\begin{lemma}
	\label{thm:combin-typeb-coeff}
	For positive integers $n$, we have  
	\begin{equation}
	\label{eqn:main_lemma1_B} 
	\sum_{ (A,\epsilon) \in \sgnbinom{[n]}{r}} q^{\inv_B(f([[0,n]-A], (A,\epsilon)))}= \qbinom{n}{r} (1+q^n)(1+q^{n-1})\cdots(1+q^{n-r+1}) .
	\end{equation}
\end{lemma}
\begin{proof}
We proceed by induction on $n$. The base case when $n=1$ is easy to verify.
We assume the result is true for $n$ and want to show it holds for $n+1$.
Thus, we want to show that
	\begin{equation}
	\label{eqn:;in_lem_inv_B} 
	\sum_{ (A,\epsilon) \in \sgnbinom{[n+1]}{r+1}} q^{\inv_B(f([[0,n+1]-A],(A,\epsilon)))}= \qbinom{n+1}{r+1} (1+q^{n+1})(1+q^{n})\cdots(1+q^{n-r+1}).
	\end{equation}
	Let $\eta(n,r)=(1+q^n)\cdots(1+q^{n-r+1})$.
	We partition $\sgnbinom{[n+1]}{r+1}$ into the disjoint union of the 
	following three subsets and determine the contribution of each of these three sets. 
	\begin{enumerate}
		\item $\mathcal{A}_1= \lbrace (A,\epsilon) \in \sgnbinom{[n+1]}{r+1}; n+1 \in (A,\epsilon)\rbrace$,
		
		\item $\mathcal{A}_2= \lbrace (A,\epsilon) \in \sgnbinom{[n+1]}{r+1}; \overline{n+1} \in (A,\epsilon)\rbrace$,
		
		\item $\mathcal{A}_3= \lbrace (A,\epsilon) \in \sgnbinom{[n+1]}{r+1}; n+1 \notin (A,\epsilon)\rbrace$.
	\end{enumerate}

	If $n+1 \in (A,\epsilon)$, as $[(A,\epsilon)]$ is in ascending order, it will 
	be the rightmost element of $f([[0,n+1]-A],[(A,\epsilon)])$ and thus 
	it will contribute no extra inversions. Thus
	\begin{equation}
	\label{eqn:lemma1B1} 
	\sum_{  (A,\epsilon) \in \mathcal{A}_1  } q^{\inv_B(f([[0,n+1]-A],[(A,\epsilon)]))}= \eta(n,r) \qbinom{n}{r}. 
	\end{equation}

	If $\overline{n+1} \in (A,\epsilon)$, then $\overline{n+1}$ has to be in 
	the '$n-r+1$'-th position in $f([[0,n+1]-A],(A,\epsilon))$. 
	Every element of $[[0,n+1]-A]$ will be to its left and will 
	thus contribute $2$ inversions.  Further, every element to its right 
	will contribute $1$ inversion. Thus, we get $2n-r+1$ new inversions. 
	Therefore, 
	
	\begin{equation}
	\label{eqn:lemma1B2} 
	\sum_{  (A,\epsilon) \in \mathcal{A}_2  } q^{\inv_B(f([[0,n+1]-A],[(A,\epsilon)]))}
	=\eta(n,r) q^{2n-r+1} \qbinom{n}{r} .
	\end{equation}
	
	Lastly, when $n+1 \in [0,n+1]-A$, then it has to be the 
	rightmost element in $[0,n+1]-A$. Every element of $(A,\epsilon)$ will 
	contribute one inversion and thus we get `$r+1$' extra inversions. 
	Hence,
	
	\begin{equation}
	\label{eqn:lemma1B3} 
	\sum_{  (A,\epsilon) \in \mathcal{A}_3  } q^{\inv_B(f([[0,n+1]-A],(A,\epsilon)))}
	= q^{r+1} \eta(n,r+1)  \qbinom{n}{r+1} 
	= q^{r+1}(1+q^{n-r}) \eta(n,r) \qbinom{n}{r+1}. 
	\end{equation}

	Summing up \eqref{eqn:lemma1B1}, \eqref{eqn:lemma1B2} and 
	\eqref{eqn:lemma1B3}, we get 
	\begin{eqnarray*}
		&&\sum_{ (A,\epsilon) \in \sgnbinom{[n+1]}{r+1}} q^{\inv_B(f([[0,n+1]-A],(A,\epsilon)))}\\
		&=& \eta(n,r) \bigg(\qbinom{n}{r} + q^{2n-r+1} \qbinom{n}{r} +q^{r+1}(1+q^{n-r}) \qbinom{n}{r+1}  \bigg)\\
		&=& \eta(n,r) \bigg((1+q^{n+1})\qbinom{n+1}{r+1}\bigg) = \eta(n+1,r+1)\qbinom{n+1}{r+1}. 
	\end{eqnarray*}
	The last equation follows from the $q$-Pascal recurrence 
	for the Gaussian binomial coefficients (see 
	\cite[Chapter 6]{petersen-eulerian-nos-book}). 
The proof of \eqref{eqn:;in_lem_inv_B} and hence of Lemma
\ref{thm:combin-typeb-coeff}  is complete.
\end{proof}

\begin{corollary}
	\label{cor:afterlemma1typeB} 
	Let $\sigma \in \BB_{n-r}$ be a signed permutation and 
	$(A,\epsilon)\in \sgnbinom{[n]}{r}$ be a signed subset. Then 
	\begin{equation}
	\label{eqn:relan_f_qbinom}
	\sum_{ (A,\epsilon) \in \sgnbinom{[n]}{r}} q^{\inv_B(f(\sigma,(A,\epsilon)))}
	= q^{\inv_B(\sigma)}\qbinom{n}{r} (1+q^n)(1+q^{n-1})\cdots(1+q^{n-r+1}).
	\end{equation} 
\end{corollary}
\begin{proof}
	For $\sigma \in \BB_{n-r}$ and $(A,\epsilon) \in \sgnbinom{[n]}{r}$, we have
	\begin{equation*}
	\inv_B(f(\sigma,(A,\epsilon))) =  \inv_B(h(\sigma),[(A,\epsilon)])) 
	=  \inv_B(f([[0,n]-A],(A,\epsilon)))+\inv_B(\sigma).
	\end{equation*} 
	The proof follows as it takes exactly $\inv_B(\sigma)$ 
	inversions to get $h(\sigma)$ from the  
	identity permutation in $\BB_{n-r}$  (recall
	$h(\sigma)$ is defined in \eqref{eqn:defn_h}).
\end{proof}

Adding 
\eqref{eqn:relan_f_qbinom} over all $\pi \in \BB_{n-r}$ 
gives us the following.

\begin{corollary}
	\label{cor:afterlemmacor2-B}
	For positive integers $n$, we have
	\begin{equation}
	\sum_{\sigma \in \BB_{n-r}}\sum_{ (A,\epsilon) \in \sgnbinom{[n]}{r}} t^{\odes_B(\sigma)}s^{\edes_B(\sigma)} q^{\inv_B(f(\sigma,(A,\epsilon)))}= B_{n-r}(s,t,q) \qbinom{n}{r} (1+q^n)\cdots(1+q^{n-r+1}). 
	\end{equation}
\end{corollary}

%
%

Reiner in \cite{reiner-distrib-descents-len-coxeter} gave the following
egf for the polynomial enumerating descents and length in $\BB_n$.  

\begin{theorem}[Reiner]
	\label{thm:reiner_typeb-stan}
	We have the following.
	\begin{equation*}
	\sum_{n \geq 0} B_n(t,q) \frac{u^n}{B_n(1,q)} = \frac{(1-t)\exp_B{(u(1-t);q}) } 
	{1 - t\exp({u(1-t);q}) }
	\end{equation*}
\end{theorem}

It can be seen that 
Theorem \ref{thm:reiner_typeb-stan} is equivalent to the following.

\begin{theorem}[Reiner]
	\label{thm:type-b-foata-desarm}
	For positive integers $n$, the polynomials  $B_n(q,t)$ satisfy the following.
	\begin{equation}
	\label{eqn:type-B-ver-foata-desarm}
	\frac{B_n(t,q)}{B_n(1,q)} = t \sum_{k=0}^n 
	\frac{B_{n-k}(t,q)(1-t)^k}{B_{n-k}(1,q) [k]_q!} + 	\frac{(1-t)^{n+1}}{B_n(1,q)}.
	\end{equation}
\end{theorem}

We are now interested in proving a trivariate analogue of Theorem \ref{thm:type-b-foata-desarm}. Towards that, 
we start with the following lemma.

\begin{lemma}
\label{lem:passing_Gni_to_Gniminus1}
Let $n$ be a positive integer and let $0\le i \le n$.
	When $i$ is odd, we have
	\begin{eqnarray}
	\sum_{\pi'\in G_{n,i}} t^{\odes_B(\pi')}s^{\edes_B(\pi')}q^{\inv_B(\pi')} =&&
	t\frac{B_{i}(s,t,q)B_n(1,q)}{B_{i}(1,q)[n-i]_q!}\nonumber\\
	&+& (1-t) \bigg\{ \sum_{\pi' \in G_{n,i-1}} t^{\odes_B(\pi')}
	s^{\edes_B(\pi')}q^{\inv_B(\pi')}\bigg\} \label{eqn:passing_Gni_to_Gniminus1iodd-B} .
	\end{eqnarray} 
	When $i$ is even, we have
	\begin{eqnarray}
	\sum_{\pi'\in G_{n,i}} t^{\odes_B(\pi')}s^{\edes_B(\pi')}q^{\inv_B(\pi')} &=&
	s\frac{B_{i}(s,t,q)B_n(1,q)}{B_{i}(1,q)[n-i]_q!}\nonumber\\&&+ (1-s) \bigg\{ \sum_{\pi' \in G_{n,i-1}} t^{\odes_B(\pi')}
	s^{\edes_B(\pi')}q^{\inv_B(\pi')}\bigg\}. \label{eqn:passing_Gni_to_Gniminus1ieven-B} 
	\end{eqnarray} 
\end{lemma}

\begin{proof} We prove \eqref{eqn:passing_Gni_to_Gniminus1iodd-B} 
first and therefore take $i$ to be odd. 
Let $F_{n,i}=G_{n,i}-G_{n,i-1}.$
	We have 
	\begin{eqnarray}
	& & \displaystyle  \sum_{(\pi,(A,\epsilon))\in \BB_i\times \sgnbinom{[n]}{n-i}} t^{\odes_B(\pi)}s^{\edes_B(\pi)}q^{\inv_B(f(\pi,(A,\epsilon)))} \nonumber  \\
	&  = & \displaystyle  \sum_{(\pi,(A,\epsilon))\in f^{-1}(G_{n,i})} t^{\odes_B(\pi)}s^{\edes_B(\pi)} q^{\inv_B(f(\pi,(A,\epsilon)))} \nonumber \\
	& = &  \displaystyle \sum_{(\pi,(A,\epsilon))\in f^{-1}(G_{n,i-1})} t^{\odes_B(\pi)}s^{\edes_B(\pi)}q^{\inv_B(f(\pi,(A,\epsilon)))} \nonumber  \\
	& & +  \displaystyle  \sum_{(\pi,(A,\epsilon))\in f^{-1}(F_{n,i})} t^{\odes_B(\pi)}s^{\edes_B(\pi)}q^{\inv_B(f(\pi,(A,\epsilon)))} \nonumber\\
	& =& \sum_{f(\pi,(A,\epsilon))\in G_{n,i-1}} t^{\odes_B(f(\pi,A))}s^{\edes_B(f(\pi,A))}q^{\inv_B(f(\pi,(A,\epsilon)))} \nonumber \\
	& & + \frac{1}{t}\bigg\{ \sum_{f(\pi,(A,\epsilon))\in F_{n,i}} t^{\odes_B(f(\pi,A))}s^{\edes_B(f(\pi,A))} q^{\inv_B(f(\pi,(A,\epsilon)))} \bigg\} \nonumber \\
	& =& \sum_{\pi'\in G_{n,i-1}} t^{\odes_B(\pi')}s^{\edes_B(\pi')}q^{\inv_B(\pi')} \nonumber\\  & &   + \frac{1}{t}\bigg\{ \sum_{\pi'\in G_{n,i}} t^{\odes_B(\pi')}s^{\edes_B(\pi')}q^{\inv_B(\pi')} - \sum_{\pi' \in G_{n,i-1}} t^{\odes_B(\pi')}s^{\edes_B(\pi')}q^{\inv_B(\pi')} \bigg\}\nonumber.
	\end{eqnarray}

The second equality follows because $f$ is a bijection between 
$\BB_i \times \sgnbinom{[n]}{n-i}$ and $G_{n,i}$. For the fourth 
equality, we have used that $i$ is odd. In the fifth equality, we are 
again using that $f$ is a bijection and $F_{n,i}=G_{n,i}-G_{n,i-1}.$ 

From Corollary \ref{cor:afterlemmacor2-B} with $i=n-r$, we have 
\begin{eqnarray}
&&B_{i}(s,t,q) \qbinom{n}{n-i} (1+q^n)\cdots(1+q^{i+1}) \nonumber \\
&=&(t-1) \sum_{\pi' \in G_{n,i-1}} t^{\odes(\pi')}
s^{\edes_B(\pi')}q^{\inv_B(\pi')} + 
\sum_{\pi'\in G_{n,i}} t^{\odes_B(\pi')}s^{\edes_B(\pi')}q^{\inv_B(\pi')}.
\label{eqn:first_in_combination}
\end{eqnarray}
The following result is easy to see
\begin{equation}
\label{eqn:easy_relation_typeb}
\qbinom{n}{n-i}(1+q^n)\cdots(1+q^{i+1})=
\frac{B_n(1,q)}{B_{i}(1,q)[n-i]_q!}. 
\end{equation}

Combining
\eqref{eqn:first_in_combination} and 
\eqref{eqn:easy_relation_typeb}
completes the proof of \eqref{eqn:passing_Gni_to_Gniminus1iodd-B}. 
The proof when $i$ is even is similar and hence is omitted.
\end{proof} 

We are now in a position to give a refinement of 
Theorem \ref{thm:type-b-foata-desarm}.

\begin{theorem}
	\label{thm:typeb_q_rec}
Let $B_0(s,t,q)=1$. When $n \geq 1$, the polynomials $B_n(s,t,q)$ 
satisfy the following recurrence.
	\begin{eqnarray}
	\label{eqn:typeb_rec_even}
	\frac{B_{n}(s,t,q)}{B_{n}(1,q)}& =& \frac{(1-t)^{k}(1-s)^{k+1}}{B_{n}(1,q)}+\sum_{r=0}^{ k-1}  t(1-t)^r(1-s)^{r+1} \frac{B_{n-2r-1}(s,t,q)}{B_{n-2r-1}(1,q)[2r+1]_q!} \nonumber\\
	& & + \sum_{r=0}^{k} s(1-t)^{r}(1-s)^{r}\frac{B_{n-2r}(s,t,q)}{B_{n-2r}(1,q)[2r]_q!} \hspace{1 cm} \mbox{if $n=2k$ is even,}  \\
	\label{eqn:typeb_rec_odd}
	\frac{B_{n}(s,t,q)}{B_n(1,q)} & =& \frac{(1-t)^{k+1}(1-s)^{k+1}}{B_n(1,q)} + \sum_{r=0}^{k} s(1-t)^{r+1}(1-s)^{r} \frac{B_{n-2r-1}(s,t,q)}{B_{n-2r-1}(1,q)[2r+1]_q!} \nonumber \\ 
	& & + \sum_{r=0}^{k} t(1-t)^{r}(1-s)^r\frac{B_{n-2r}(s,t,q)}{B_{n-2r}(1,q)[2r]_q!} \hspace{1 cm} \mbox{if $n=2k+1$ is odd}. 
	\end{eqnarray}
\end{theorem}

\begin{proof}
Recall that $G_{n,i}$ is the set of signed permutations whose rightmost 
$(n-i)$ entries form an increasing run.  Thus, $G_{n,n-1} = \BB_n$. 
Further, $G_{n,-1}$ is just the signed permutation $\pi=0,1,2, \cdots, n.$
Let $n$ be even. By repeatedly applying 
\eqref{eqn:passing_Gni_to_Gniminus1ieven-B} and 
\eqref{eqn:passing_Gni_to_Gniminus1iodd-B} we have 
\begin{eqnarray*}
B_{n}(s,t,q)	&=& \sum_{\pi \in G_{n,n-1}} t^{\odes_B(\pi)}s^{\edes_B(\pi)}q^{\inv_B(\pi)} \\
&=&t\frac{B_{n-1}(s,t,q)B_n(1,q)}{B_{n-1}(1,q)[1]_q!}+(1-t)\Bigg(\sum_{\pi \in G_{n,n-2}}  t^{\odes_B(\pi)}s^{\edes_B(\pi)}q^{\inv_B(\pi)}\Bigg) \\
&=&t\frac{B_{n-1}(s,t,q)B_n(1,q)}{B_{n-1}(1,q)[1]_q!} +(1-t) s\frac{B_{n-2}(s,t,q)B_n(1,q)}{B_{n-2}(1,q)[2]_q!}\\& & +(1-t)(1-s)\Big(\sum_{\pi \in G_{n,n-3}} t^{\odes_B(\pi)}s^{\edes_B(\pi)}q^{\inv_B(\pi)}\Big)  \\
& = & 
(1-t)^{k}(1-s)^{k+1}+\sum_{r=0}^{ k-1}  t(1-t)^r(1-s)^{r+1} \frac{B_{n-2r-1}(s,t,q)B_n(1,q)}{B_{n-2r-1}(1,q)[2r+1]_q!} \nonumber\\
& & + \sum_{r=0}^{k} s(1-t)^{r}(1-s)^{r}\frac{B_{n-2r}(s,t,q)B_n(1,q)}{B_{n-2r}(1,q)[2r]_q!}.
\end{eqnarray*}	
This completes the proof of \eqref{eqn:typeb_rec_even}. We now 
consider the case when $n$ is odd. Here, we will get
\begin{eqnarray*}
B_n(s,t,q) &=& \sum_{\pi \in G_{n,n-1}} t^{\odes_B(\pi)}s^{\edes_B(\pi)}q^{\inv_B(\pi)} \\
&=&s\frac{B_{n-1}(s,t,q)}{B_{n-1}(1,q)[1]_q!}+(1-s)\Bigg(\sum_{\pi \in G_{n,n-2}}  t^{\odes_B(\pi)}s^{\edes_B(\pi)}q^{\inv_B(\pi)}\Bigg) \\
&=&s\frac{B_{n-1}(s,t,q)B_n(1,q)}{B_{n-1}(1,q)[1]_q!}+(1-s) t\frac{B_{n-2}(s,t,q)}{B_{n-2}(1,q)[2]_q!}+ \\
& & +(1-s) (1-t)\Big(\sum_{\pi \in G_{n,n-3}} t^{\odes_B(\pi)}s^{\edes_B(\pi)}q^{\inv_B(\pi)}\Big). 
\end{eqnarray*}	
Continuing as in the case when $n$ was even, completes the 
proof of \eqref{eqn:typeb_rec_odd} and hence completes the proof 
of Theorem \ref{thm:typeb_q_rec}.
\end{proof}

\subsection{Type B Generating Functions}
\label{subsec:type_b_gen_fns}
We recast Theorem \ref{thm:typeb_q_rec} in the language of 
egfs to prove Theorem \ref{thm:from_carlitz_scoville_typeb_eulerian_q-anal}.
Recall our definitions from
\eqref{eqn:type-b-defns-odd-even-biv-eulerian}.

%

\begin{proof}[Proof of Theorem 
	\ref{thm:from_carlitz_scoville_typeb_eulerian_q-anal}]
	For positive integers $n=2k$, we have 
\begin{eqnarray}
\label{eqn:bsum1}
\frac{B_{n}(s,t,q)u^{2k}}{B_{n}(1,q)}& =& 
\frac{(1-t)^{k}(1-s)^{k+1}u^{2k}}{B_{n}(1,q)}+ 
\sum_{r=0}^{k} \bigg(\frac{(1-t)^{r}(1-s)^{r}u^{2r}}{[2r]_q!}\bigg)
\bigg(\frac{sB_{n-2r}(s,t,q)u^{n-2r}}{B_{n-2r}(1,q)}\bigg) \nonumber \\ 
& & + \sum_{r=0}^{ k-1} \bigg(\frac{(1-t)^r(1-s)^{r+1}u^{2r+1}}{[2r+1]_q!}\bigg)
\bigg(\frac{ tB_{n-2r-1}(s,t,q)u^{n-2r-1}}{B_{n-2r-1}(1,q)}\bigg).
\end{eqnarray} 
When $n=2k+1$, we have  
\begin{eqnarray} 
\label{eqn:bsum2}
\frac{B_{n}(s,t,q)u^{2k+1}}{B_n(1,q)} & =& 
\frac{(1-t)^{k+1}(1-s)^{k+1}u^{2k+1}}{B_n(1,q)} \nonumber \\ 
& &  +  \sum_{r=0}^{k} \bigg(\frac{(1-t)^{r}(1-s)^r u^{2r}}{[2r]_q!}\bigg)
\bigg(\frac{tB_{n-2r}(s,t,q)u^{n-2r}}{B_{n-2r}(1,q)}\bigg) \nonumber  \\ 
& &+  \sum_{r=0}^{k} \bigg(\frac{(1-t)^{r+1}(1-s)^{r}u^{2r+1}}{[2r+1]_q!}\bigg)
\bigg( \frac{sB_{n-2r-1}(s,t,q)u^{n-2r-1}}{B_{n-2r-1}(1,q)}\bigg) . 
\end{eqnarray}
Summing \eqref{eqn:bsum1},\eqref{eqn:bsum2} over $k\geq 0$ yields 
	\begin{eqnarray}
	\label{eqn:typeb-egf-solve1}
	(1-s)\cosh_B(Mu;q)+M\sinh_B(Mu;q) & = & \mathcal{B}_0\bigg(1-s\cosh_q(Mu)-\frac{s\sinh_q(Mu)}{L}\bigg)\nonumber \\
	& & +\mathcal{B}_1(1-t\cosh_q(Mu)-tL\sinh_q(Mu)).
	\end{eqnarray}
	where $L=\sqrt{(1-s)/(1-t)}$, $\mathcal{B}_0=H_0(s,t,q,u)$ and $\mathcal{B}_1=H_1(s,t,q,u).$
	Changing $u$ to $-u$ gives us
\begin{eqnarray}
\label{eqn:typeb-egf-solve2}
	(1-s)\cosh_B(Mu;q)-M\sinh_B(Mu;q)& =& \mathcal{B}_0\bigg(1-s\cosh_q(Mu)+\frac{s\sinh_q(Mu)}{L}\bigg)\nonumber \\ & & -\mathcal{B}_1(1-t\cosh_q(Mu)+tL\sinh_q(Mu)).
	\end{eqnarray}
	Solving 
\eqref{eqn:typeb-egf-solve1} and 
\eqref{eqn:typeb-egf-solve2}, completes the proof.
\end{proof}

\blue{
\begin{remark}
\label{rem:equiv_type-b}
We show that setting $q=1$ in Theorem 
\ref{thm:from_carlitz_scoville_typeb_eulerian_q-anal} gives 
Theorem \ref{thm:pan_zeng_biv_eul_egf}. 
We claim that 
$H_0(s,t,1,u)=H_0(s,t,\frac{u}{2})$ and likewise
$H_1(s,t,1,u)=H_1(s,t,\frac{u}{2}).$  
As $B_n(1,1)=2^nn!$, $\cosh_B(u;1)=\cosh(\frac{u}{2})$, 
$\sinh_B(u;1)=\sinh(\frac{u}{2})$, and 
$\mathrm{e}_q(u)\mathrm{e}_q(-u)|_{q=1}=1$, 
setting $q=1$ on the right hand side of 
\eqref{eqn:from_carlitz_scoville_typeb_eulerian_q-anal_even} 
gives the right 
hand side of \eqref{eqn:typeb-eventerms-pz}. 
Similarly, setting $q=1$ on the right hand side of 
\eqref{eqn:from_carlitz_scoville_typeb_eulerian_q-anal_odd}, 
we get the right hand side of \eqref{eqn: typeb-oddterms-pz}. 
\end{remark}
}

%
%
%

We are now in a position to prove Theorem
\ref{thm:from_carlitz_scoville_typeb_alt_eulerian_q-anal}.

\begin{proof}[Proof of Theorem 
\ref{thm:from_carlitz_scoville_typeb_alt_eulerian_q-anal}]
	If $B_{2k}(s,t,q)$ is the polynomial defined in 
\eqref{eqn:defn_eulerian_bi_tri_variant}, it is easy to see that  
	$\displaystyle \hat{B}_{2k}(s,t,q)  = 
	s^{k} B_{2k}(1/s,t,q)$.
	Therefore, 
	\begin{eqnarray}
	\hat{H}_0(s,t,q,u)&=&H_0\Big(\frac{1}{s},t,q,\sqrt{s}u\Big)\nonumber\\
	&=& \frac{(s-1)\bigg( (1-t\cos_q(Mu))\cos_{B}(Mu;q)-t\sin_q(Mu)\sin_{B}(Mu;q)\bigg)}{s+t\mathrm{e}_q(iMu)\mathrm{e}_q(-iMu)-(ts+1)\cos_q(Mu)}.
	\end{eqnarray}
As the proof is complete, we move on to the case when $n=2k+1$ is odd.  
Clearly, in this case, we have 
$\hat{B}_{2k+1}(s,t,q)
	=  s^{k+1}B_{2k+1}\bigg(\frac{1}{s},t,q \bigg).$
	Therefore,
	\begin{eqnarray}
	\hat{H}_1(s,t,q,u) &=& \sqrt{s}H_1(\frac{1}{s},t,q,\sqrt{s}u)\nonumber\\
	&=& \frac{-M\bigg((s-\cos_q(Mu)\sin_{B}(Mu;q)+\sin_q(Mu)\cos_{B}(Mu;q) \bigg)}{s+t\mathrm{e}_q(iMu)\mathrm{e}_q(-iMu)-(ts+1)\cos_q(Mu)}.
	\end{eqnarray}
	This completes the proof.
\end{proof}

\begin{corollary}
	We have the following egf for the type $\BB$ bivariate alternating descent polynomials:
	\begin{equation}
	\sum_{n\geq 0}\hat{B}_n(s,t) \frac{u^n}{n!}=\frac{-(s-1)(t-1)\cos(Mu)-M(s+1)\sin(Mu)}{s+t-(ts+1)\cos(2Mu)}.
	\end{equation}
\end{corollary}
\begin{corollary}
	We get an alternate proof of the following egf for the type B 
	alternating descent polynomials (see also 
	\cite{ma-fang-mansour-yeh-alt_eul_poly_left_pk} and 
	\cite{pan-new_combin_formula_alt_des}):
	\begin{equation}
	\sum_{n\geq 0}\hat{B}_n(t) \frac{u^n}{n!}=\frac{-(t-1)^2\cos((1-t)u)+(t^2-1)\sin((1-t)u)}{2s-(t^2+1)\cos(2(1-t)u)}.
	\end{equation}
\end{corollary}

\blue{As mentioned in Section \ref{sec:intro} though we consider a two variable
enumerator, we can get a four variable version and hence a type B 
counterpart of Theorem \ref{thm:carlitz_scoville}.   }
\blue{
Define variables $s_0$, $t_0$, $s_1$ and $t_1$ 
to keep track of even ascents, even descents, 
odd ascents and odd descents respectively.  
Let  $m=\sqrt{(s_0-t_0)(s_1-t_1)}$.  
Define the five 
variable distribution 
$$B_n(s_0,s_1,t_0,t_1,q)=\sum_{w\in \BB_n} 
s_0^{\eascb(w)}s_1^{\oasc_{B}(w)}t_0^{\edesb(w)}
t_1^{\odesb(w)}q^{\invb(w)}.$$ 
Further, define the generating functions 
\begin{eqnarray}
\label{eqn:defn_q_egf-type-b-five-var-even}
H_0(s_0,s_1,t_0,t_1,q,u)& =& \sum_{k\geq 0}B_{2k}(s_0,s_1,t_0,t_1,q)\frac{u^{2k}}{B_{2k}(1,q)}, \\
\label{eqn:defn_q_egf-type-b-five-var-odd}
H_1(s_0,s_1,t_0,t_1,q,u)& =& \sum_{k\geq 0}B_{2k+1}(s_0,s_1,t_0,t_1,q)\frac{u^{2k+1}}{B_{2k+1}(1,q)}.
\end{eqnarray}
}

\blue{
\begin{theorem}
\label{q-egf-type-b-counterpart-pan-zeng}
We have the egfs
\begin{eqnarray*}
\frac{H_0(s_0,s_1,t_0,t_1,q,u)}{s_0-t_0} & = & \frac{\bigg(\big(s_1-t_1\cosh_q(mu)\big){\cosh_B(mu;q)}+t_1\sinh_q(mu){\sinh_B(mu;q)}\bigg)}{s_0s_1-(t_0s_1+s_0t_1)\cosh_q(mu)+t_0t_1\mathrm{e}_q(mu)\mathrm{e}_q(-mu)}, 
\\
H_1(s_0,s_1,t_0,t_1,q,u) & = & \frac{m\bigg(\big(s_0-t_0\cosh_q(mu)\big){\sinh_B(mu;q)}+t_0\sinh_q(mu){\cosh_B(mu;q)}\bigg)}{s_0s_1-(s_1t_0+t_1s_0)\cosh_q(mu)+t_0 t_1\mathrm{e}_q(mu)\mathrm{e}_q(-mu)}. 
\end{eqnarray*}
\end{theorem}
\begin{proof}
Recalling \eqref{eqn:type-b-defns-odd-even-biv-eulerian}, 
it is easy to see that  
\begin{eqnarray*}
H_0(s_0,s_1,t_0,t_1,q,u)& = & 
H_0\Bigg(\frac{t_0}{s_0},\frac{t_1}{s_1},q,\sqrt{s_0s_1}u\Bigg), \\
H_1(s_0,s_1,t_0,t_1,q,u)& = & 
\frac{\sqrt{s_0}}{\sqrt{s_1}}H_1\Bigg(\frac{t_0}{s_0},
\frac{t_1}{s_1},q,\sqrt{s_0s_1}u\Bigg).
\end{eqnarray*}
The proof is complete.
\end{proof}
}

\subsection{$q$-analogues of Hyatt's recurrences and symmetries}
\label{subsec:hyatt_type_b_q_anal}
Hyatt gives a proof of Theorem \ref{thm:hyatt_recurrences} for 
the polynomials $B_n^+(t)$ by considering a statistic 
$\maxdrop_B$.  We give an inclusion exclusion argument.

\begin{proof}[Proof of Theorem \ref{thm:typeb_plus}]
	We prove the two recurrences separately. 
	Let $\hat{A}_k$ be the set of signed permutations in $\BB_n$ such 
	that the last $k+1$ elements are positive and are arranged in 
	descending order.  Thus, the set $\hat{A}_{k-1}-\hat{A}_k$ is
	the set of signed permutations with no descent in the $(n-k)$-th 
	position and with their last $k$ elements being positive and descending. 
	Define 
$$A_k(s,t,q)=\sum_{w \in \hat{A}_k} t^{\odesb(w)}s^{\edesb(w)}q^{\invb(w)}.$$
We abbreviate $A_k(s,t,q)$ as $A_k$ for the rest of this proof for 
better readability.  When $n$ is even, we claim that
	\begin{eqnarray}
	\label{eqn:typeb_even_even}
	q^{\binom{2r}{2}}\qbinom{n}{2r} B_{n-2r}(s,t,q)s^rt^r & = & sA_{2r-1}-(s-1)A_{2r},\\
	\label{eqn:typeb_even_odd}
	q^{\binom{2r+1}{2}}\binom{n}{2r+1}_qB_{n-2r-1}(s,t,q)s^rt^{r+1}& = &tA_{2r}-(t-1)A_{2r+1}.
	\end{eqnarray}
	When $n$ is odd, we claim that
	\begin{eqnarray}
	\label{eqn:typeb_odd_even}
	q^{\binom{2r}{2}}\binom{n}{2r}_qB_{n-2r}(s,t,q)s^{r}t^r& =& tA_{2r-1}-(t-1)A_{2r},\\
	\label{eqn:typeb_odd_odd}
	q^{\binom{2r+1}{2}}\binom{n}{2r+1}_qB_{n-2r-1}(s,t,q)s^{r+1}t^{r}& =& sA_{2r}-(s-1)A_{2r+1}.
	\end{eqnarray}
	We only prove \eqref{eqn:typeb_even_even}. The proofs of 
	\eqref{eqn:typeb_even_odd}, \eqref{eqn:typeb_odd_even} and 
	\eqref{eqn:typeb_odd_odd} follow from a very similar argument.
	Recall that $\binom{[n]}{n-i}$ is the set of all $(n-i)$-sized subsets of 
	$[n]$. Given $A\in \binom{[n]}{n-i}$, we arrange its elements in 
	descending order and list them as $a_1,a_2,\cdots,a_{n-i}$ 
	with  $a_1>a_2>\cdots>a_{n-i}>0$.
	Define a new juxtaposition map 
	$f':\BB_{i}\times \binom{[n]}{n-i}\rightarrow \hat{A}_{n-i-1}$ that 
	takes $(\psi,A)$ to the signed permutation $\psi_{[n]-A},a_1,a_2,\cdots,a_{n-i}$, i.e. $$f'(\psi,A)=\psi_{[n]-A},a_1,a_2,\cdots,a_{n-i}.$$
	
	It is easy to see that $f'$ is a bijection from 
	$\BB_{i}\times \binom{[n]}{n-i}$ to $\hat{A}_{n-i-1}$.
	We define $\invb([X],[Y])$ to be the number of inversions 
	that occur between the $X$ and $Y$.
	The LHS of \eqref{eqn:typeb_even_even} is clearly obtained as follows
	\begin{eqnarray*}
		&&\sum_{((\psi,A)\in \BB_{n-2r})\times \binom{[n]}{2r}} t^{\odes_B(\psi)+\odes_B(A)}s^{\edes_B(\psi)+\edes_B(A)}q^{\invb(\psi)+\invb(A)+\invb([\psi],[A])}\\ 
		&=&q^{\binom{2r}{2}}s^{r-1}t^r \sum_{(\psi,A)\in \BB_{n-2r}\times \binom{[n]}{2r}} t^{\odes_B(\psi)}s^{\edes_B(\psi)}q^{\invb(\psi)+\invb([\psi],[A])}\\
		&=& q^{\binom{2r}{2}}s^{r-1}t^r \sum_{\psi \in \BB_{n-2r}}\sum_{A\in \binom{[n]}{2r}}t^{\odes_B(\psi)}s^{\edes_B(\psi)}q^{\invb(\psi)+\invb([\psi],[A])}\\
		&=&q^{\binom{2r}{2}}s^{r-1}t^r \qbinom{n}{2r} B_{n-2r}(s,t,q).
	\end{eqnarray*}
	
	The expression above does not account for the descent occurring 
	at the $(n-2r)$-th position.  Thus, it is off by a factor of 
	$\frac{1}{s}$ on the set $\hat{A}_{2r}$.  Further,  it counts 
	correctly on the set $\hat{A}_{2r-1}-\hat{A}_{2r}$.
	This gives us
	\begin{equation*}
	q^{\binom{2r}{2}}s^{r-1}t^r \qbinom{n}{2r} B_{n-2r}(s,t,q)=A_{2r-1}-A_{2r}+\frac{1}{s}A_{2r},
	\end{equation*}
	which is equivalent to \eqref{eqn:typeb_even_even}.
	Similarly for $2r+1$ , we get \eqref{eqn:typeb_even_odd}
	\begin{equation*}
	q^{\binom{2r+1}{2}} s^rt^{r+1} \qbinom{n}{2r+1} B_{n-2r-1}(s,t,q)=tA_{2r}-(t-1)A_{2r+1}.
	\end{equation*}
	
	Equations \eqref{eqn:typeb_even_even} and 
	\eqref{eqn:typeb_even_odd} gives 
	\begin{eqnarray}
	& &   q^{\binom{2r}{2}}\qbinom{n}{2r} B_{n-2r}(s,t,q)(s-1)^{r-1} (t-1)^r \nonumber \\
	& & =  \bigg(\frac{s-1}{s}\bigg)^{r-1}\bigg(\frac{t-1}{t}\bigg)^rA_{2r-1} -\bigg(\frac{s-1}{s}\bigg)^{r}\bigg(\frac{t-1}{t}\bigg)^{r}A_{2r}, 
	\label{eqn:typeb_even_wantedform} \\
	& &  q^{\binom{2r+1}{2}}\qbinom{n}{2r+1} B_{n-2r-1}(s,t,q)(s-1)^r (t-1)^r \nonumber \\
	\label{eqn:typeb_odd_wantedform}
	& & =  \bigg(\frac{s-1}{s}\bigg)^{r}\bigg(\frac{t-1}{t}\bigg)^rA_{2r} -\bigg(\frac{s-1}{s}\bigg)^{r}\bigg(\frac{t-1}{t}\bigg)^{r+1}A_{2r+1}.
	\end{eqnarray}	
	
	Summing \eqref{eqn:typeb_even_wantedform}, 
	over the indices  $1\leq r \leq \frac{n}{2}$
	and  \eqref{eqn:typeb_odd_wantedform} 
	over the indices $0\leq r\leq \frac{n-2}{2}$, we get 
	\begin{eqnarray*}
		A_0 & =& \sum_{r=0}^{\frac{n}{2}-1}q^{\binom{2r+1}{2}} \qbinom{n}{2r+1} B_{n-2r-1}(s,t,q)(s-1)^r (t-1)^r \\
		& & + \sum_{r=1}^{\frac{n}{2}} q^{\binom{2r}{2}} \qbinom{n}{2r} B_{n-2r}(s,t,q)(s-1)^{r-1} (t-1)^r.
	\end{eqnarray*}
As $\hat{A}_0$ 
is the set of signed permutations with elements having positive
	last element (ie $\BB_n^+$), this completes our proof.
\end{proof}

We recall the polynomials $B_n^+(s,t,q)$ and $B_n^-(s,t,q)$ 
from \eqref{defn:poly_pm_like_hyatt}.   We consider the 
map that flips the sign of all elements below and give
a few properties.

\begin{lemma}
	\label{lem:sgnflipmap} 
	Let $f: \BB_n \rightarrow \BB_n$ be the involution that sends 
	$w=w_1,\cdots,w_n$ to $\overline{w}=\overline{w_1},\cdots,
	\overline{w_n}$. Then, we have the following.
	\begin{enumerate}
		\item When $n=2k+1$, we have $\odesb(w)+\odesb(f(w))=k$ and 
		when $n=2k$, we have $\odesb(w)+\odes_B(f(w))=k$.
		\item When $n=2k+1$, we have $\edesb(w)+\edesb(f(w))=k+1$ and when 
		$n=2k$, we have $\edesb(w)+\edesb(f(w))=k$.  
		\item The sum $\invb(w)+\invb(f(w))=n^2.$
	\end{enumerate}
\end{lemma}
\begin{proof} The proof of the first two asserions are straightforward and hence omitted. For the third part, 
	we recall that  $\invb(w)=\inv(w)+\negsum(w)$ where 
	$\inv(w)= |\{(i,j): 1 \leq i< j \leq n: w_i>w_j\}|$ and $\negsum(w)=\sum_{i \in \Negs(w)}i$. 
	Thus, we have 
	\begin{eqnarray*}
		\invb(w)+\invb(f(w))    &=&\inv(w)+\negsum(w)+\inv(\overline{w})+\negsum(\overline{w})\\
		&=&\inv(w)+\inv(\overline{w})+\negsum(w)+\negsum(\overline{w}) \\
		& = &  \binom{n+1}{2}+\binom{n}{2}=n^2.
	\end{eqnarray*}
	The proof is complete.
\end{proof}

\subsection{Symmetry results}
\label{subsec:symm_type-b}

\begin{theorem}
	\label{thm:typeb_minus}
	For positive integers $n$, we have
	\begin{eqnarray*}
		B_n^-(s,t,q) &  = &  q^{n^2}s^{k+1}t^k B_n^+(s^{-1},t^{-1},q^{-1}) \mbox{ \hspace{6 mm} when $n = 2k+1$, } \\
		B_n^-(s,t,q) & = &q^{n^2}s^{k}t^k B_n^+(s^{-1},t^{-1},q^{-1}) \mbox{ \hspace{10 mm}  when $n=2k$.}
	\end{eqnarray*}
	Therefore, we have 
	\begin{eqnarray*}
		B_n(s,t,q) & = & B_n^+(s,t,q)+q^{n^2}s^{k+1}t^k B_n^+(s^{-1},t^{-1},q^{-1})
		\mbox{ \hspace{3 mm} when $n=2k+1$,}\\
		B_n(s,t,q) & =& B_n^+(s,t,q)+q^{n^2}s^{k}t^k B_n^+(s^{-1},t^{-1},q^{-1})
		\mbox{ \hspace{7 mm} when $n=2k$.}\\
	\end{eqnarray*}
\end{theorem}
\begin{proof}
	Let $f:\BB_n^+ \rightarrow \BB_n^-$ be the map 
	that sends $w=w_1,\cdots, w_n$ to $\overline{w}=\overline{w_1},\cdots,
	\overline{w_n}$. By Lemma \ref{lem:sgnflipmap}, when $n=2k$, we have
	\begin{eqnarray*}
		\sum_{w\in \BB_n^-} t^{\odesb(w)}s^{\edesb(w)}q^{\invb(w)}&=&\sum_{w\in \BB_n^+} t^{\odesb(f(w))}s^{\edesb(f(w))}q^{\invb(f(w))}\\
		&=&\sum_{w\in \BB_n^+}t^{k-\odesb(w)}s^{k-\edesb(w)}q^{n^2-\invb(w)}\\
		&=&q^{n^2}s^kt^k\sum_{w\in \BB_n^+} t^{-\odesb(w)}s^{-\edesb(w)}q^{-\invb(w)}.
	\end{eqnarray*}
	When $n=2k+1$, we have
	\begin{eqnarray*}
		\sum_{w\in \BB_n^-} t^{\odesb(w)}s^{\edesb(w)}q^{\invb(w)}&=&\sum_{w\in \BB_n^+} t^{\odesb(f(w))}s^{\edesb(f(w))}q^{\invb(f(w))}\\
		&=&\sum_{w\in \BB_n^+} t^{k-\odesb(w)}s^{k+1-\edesb(w)}q^{n^2-\invb(w)}\\
		&=&q^{n^2}s^{k+1}t^k\sum_{w\in \BB_n^+} t^{-\odesb(w)}s^{-\edesb(w)}q^{-\invb(w)}.
	\end{eqnarray*}
	completing the proof.
\end{proof}

In a similar manner, the following result also follows.
\begin{lemma}
	\label{thm: type_b_reciprocal}
	We have 
	\begin{eqnarray*}
		B_{2k}(s,t,q) & = & q^{n^2}s^kt^k B_{2k}\bigg
		(\frac{1}{s},\frac{1}{t},\frac{1}{q}\bigg) \mbox{ \hspace{13 mm}  when $n=2k$, } \\ 
		B_{2k+1}(s,t,q) & = & q^{n^2}s^{k+1}t^k B_{2k+1}\bigg(\frac{1}{s},\frac{1}{t},\frac{1}{q}\bigg) \mbox{ \hspace{5 mm} when $n=2k+1$}.
	\end{eqnarray*}
\end{lemma}

\section{Type D analogues}

Let $H_{n,i}$ be the set of signed permutations $\pi \in \DD_n$ 
such that the last $n-i$ elements of $\pi$ are increasing, that is 
we have $\pi_{i+1} < \pi_{i+2} < \dots <\pi_{n-1} < \pi_n.$ 
Clearly, $|H_{n,i}|= 2^{n-1}\binom{n}{i}i!.$ 

Let $\sigma= \sigma_1, \cdots, \sigma_{n-i} \in \DD_{n-i}$ and 
$(A,\epsilon) \in \sgnbinom{[n]}{i} $ be a signed subset. Moreover, 
let $[n]-A=\{c_1, c_2, \dots, c_{n-i}\}$ be written in ascending order.
Thus, $c_1 < c_2 < \dots < c_{n-i}.$  We define two maps 
$h: \DD_{n-i} \to \DD_{ \{c_1, c_2, \dots, c_{n-i} \} }$ and 
$h_D:\DD_{n-i} \to \BB_{c_1, c_2, \dots, c_{n-i}}-
\DD_{ \{c_1, c_2, \dots, c_{n-i} \} }$ as follows:
$$h(\sigma)= \pi_1, \pi_2, \dots, \pi_{n-i} 
\hspace{2 mm} \mbox{ and } \hspace{2 mm}
h_D(\sigma)= \overline{\pi_{1}}, \pi_2, \dots, \pi_{n-i},$$
where for $1 \leq i \leq n-i$, if $|\sigma_i|=k$ then  
$|\pi_i|=c_k$ and $\pi_i$ has the same sign as $\sigma_i$. 
Both maps $h,h_D$ are clearly bijections and hence invertible. 

If $(A,\epsilon)$ has an even number of negative elements, then by 
inverting the map $h$ on the elements of $[0,n]-A$ and appending 
the elements of $(A,\epsilon)$ in ascending order, we get a 
signed permutation in $H_{n,n-i}$.
Similarly, if $(A,\epsilon)$ has odd number of negative elements, 
then by inverting the map $h_D$ on the elements of $[0,n]-A$ and 
appending the elements of $(A,\epsilon)$ in ascending order, we 
get a signed permutation in $H_{n,n-i}$.

These maps are also invertible, so we have a bijection 
$f_{\DD}: \DD_{n-i} \times \sgnbinom{[n]}{i} \mapsto H_{n,n-i}$ 
defined as follows.  Let $\sigma \in \BB_{n-i}$ and 
$(A,\epsilon) \in \sgnbinom{[n]}{i}$. 

Define
\[f_{\DD}(\sigma,(A,\epsilon))=
 \begin{cases} 
      h(\sigma)[(A,\epsilon)] & \mbox{if $(A,\epsilon)$ has even no. of negatives }\\
      h_D(\sigma)[(A,\epsilon)] & \mbox{if $(A,\epsilon)$ has odd no. of negatives}  
   \end{cases}
\]
where $(A,\epsilon)$ is juxtaposed at the end of the $h(\sigma)$ or $h_D(\sigma)$.

We start with the following type D counterpart of Lemma 
\ref{thm:combin-typeb-coeff}.
\begin{lemma}
	\label{thm:combin-typed-coeff-new}
	Let $(A,\epsilon)\in \sgnbinom{[n]}{r}$ be a signed subset  of $[n]$. Then,
	\begin{equation}
	\sum_{ (A,\epsilon) \in \sgnbinom{[n]}{r}} q^{\invd(f_{\DD}([[n]-A],[(A,\epsilon)]))}= \qbinom{n}{r} (1+q^{n-1})(1+q^{n-2})\cdots(1+q^{n-r}) .
	\end{equation}
\end{lemma}

\begin{proof}
We proceed by induction on $n$. The base case when $n=1$ is easy.
	We assume that our Lemma is true for $n$ and show that it holds for $n+1$.
Thus, we need to show the following:
\begin{equation}
\label{eqn:;in_lem_inv_D} 
\sum_{ (A,\epsilon) \in \sgnbinom{[n+1]}{r+1}} q^{\inv_D(f_{\mathcal{D}}([[n+1]-A],(A,\epsilon)))}= \qbinom{n+1}{r+1} (1+q^{n})(1+q^{n-1})\cdots(1+q^{n-r}).
\end{equation}
Let $\eta(n,r)=(1+q^{n-1})\cdots(1+q^{n-r})$.
We partition $\sgnbinom{[n+1]}{r+1}$ into the disjoint union of the 
following three subsets.
\begin{enumerate}
\item $\mathcal{A}_1= \lbrace (A,\epsilon) \in \sgnbinom{[n+1]}{r+1}: n+1 \in (A,\epsilon)\rbrace$,
		
\item $\mathcal{A}_2= \lbrace (A,\epsilon) \in \sgnbinom{[n+1]}{r+1}: \overline{n+1} \in (A,\epsilon)\rbrace$,
		
\item $\mathcal{A}_3= \lbrace (A,\epsilon) \in \sgnbinom{[n+1]}{r+1}: n+1 \notin (A,\epsilon)\rbrace$.
\end{enumerate}
	
We next determine the contribution to 
$ \sum_{ (A,\epsilon) \in \sgnbinom{[n+1]}{r+1}} 
q^{\inv_D(f_{\DD}([[n+1]-A],(A,\epsilon)))}$ from each of the above sets. 
If $n+1 \in (A,\epsilon)$, as $[(A,\epsilon)]$ is in ascending order, it will 
be the rightmost element of $f([[n+1]-A],[(A,\epsilon)])$ and thus 
it will contribute no extra inversions. Thus
	\begin{equation}
	\label{eqn:lemma1B1-D} 
	\sum_{  (A,\epsilon) \in \mathcal{A}_1  } q^{\inv_D(f_{\DD}([[n+1]-A],[(A,\epsilon)]))}= \eta(n,r) \qbinom{n}{r}. 
	\end{equation}
	
If $\overline{n+1} \in (A,\epsilon)$, then $\overline{n+1}$ has to be in 
the $(n-r+1)$-th position in $f([[n+1]-A],(A,\epsilon))$. 
Every element of $[[n+1]-A]$ will be to its left and will 
thus contribute $2$ inversions.  Further, every element to its right 
will contribute $1$ inversion. Thus, we get $2n-r$ new inversions. 
Therefore, 
	
	\begin{equation}
	\label{eqn:lemma1B2-D} 
	\sum_{  (A,\epsilon) \in \mathcal{A}_2  } q^{\inv_D(f_{\DD}([[n+1]-A],[(A,\epsilon)]))}
	=\eta(n,r) q^{2n-r} \qbinom{n}{r} .
	\end{equation}
	
	Lastly, when $n+1 \in [n+1]-A$, then it has to be the 
	rightmost element in $[n+1]-A$. Every element of $(A,\epsilon)$ will 
	contribute one inversion and thus we get `$r+1$' extra inversions. 
	Hence,
	
	\begin{equation}
	\label{eqn:lemma1B3-D} 
	\sum_{  (A,\epsilon) \in \mathcal{A}_3  } q^{\inv_D(f_{\DD}([[n+1]-A],(A,\epsilon)))}
	= q^{r+1} \eta(n,r+1)  \qbinom{n}{r+1} 
	= q^{r+1}(1+q^{n-r-1}) \eta(n,r) \qbinom{n}{r+1}. 
	\end{equation}

	Summing up \eqref{eqn:lemma1B1-D}, \eqref{eqn:lemma1B2-D} and 
	\eqref{eqn:lemma1B3-D}, we get 
	\begin{eqnarray*}
		&&\sum_{ (A,\epsilon) \in \sgnbinom{[n+1]}{r+1}} q^{\inv_D(f_{\mathcal{D}}([[n+1]-A],(A,\epsilon)))}\\
		&=& \eta(n,r) \bigg(\qbinom{n}{r} + q^{2n-r} \qbinom{n}{r} +q^{r+1}(1+q^{n-r-1}) \qbinom{n}{r+1}  \bigg)=
		\eta(n+1,r+1)\qbinom{n+1}{r+1}. 
	\end{eqnarray*}
	The last equation follows from the $q$-Pascal recurrence 
	for the Gaussian binomial coefficients. This completes the proof. 
\end{proof}

\begin{corollary}
	\label{thm:typeb-sst-combin-D}
	Let $\sigma \in \DD_{n-r}$ be a signed permutation, $(A,\epsilon)\in \sgnbinom{[n]}{r}$ be a signed subset. 
	\begin{equation}
	\sum_{ (A,\epsilon) \in \sgnbinom{[n]}{r}} q^{\invd(f_{\DD}(\sigma,[(A,\epsilon)]))}= q^{\invd(\sigma)}\qbinom{n}{r} (1+q^{n-1})(1+q^{n-2})\cdots(1+q^{n-r}).
	\end{equation} 
\end{corollary}

\begin{proof}
The proof follows exactly as the proof of Corollary \ref{cor:afterlemma1typeB}. 
The result follows by noting that changing the sign of the 
first element does not affect the type D inversion statistic.
\end{proof}

\begin{corollary}
	\label{thm:cor-typed-coeff}
	We have
	\begin{eqnarray*}
		& &    \sum_{\sigma \in  \DD_{n-r}}\sum_{ (A,\epsilon) \in \sgnbinom{[n]}{r}} t^{\odesd(\sigma)}s^{\edesd(\sigma)} q^{\invd(f_{\DD}(\sigma,[(A,\epsilon)]))} \\
		& & = D_{n-r}(s,t,q) \qbinom{n}{r} (1+q^{n-1})(1+q^{n-2})\cdots(1+q^{n-r}).
	\end{eqnarray*}
\end{corollary}
\begin{proof}
\blue{
    For a particular $\sigma\in \DD_{n-r}$, we have 
 \begin{eqnarray*}
 &&t^{\odesd(\sigma)}s^{\edesd(\sigma)}\sum_{ (A,\epsilon) \in \sgnbinom{[n]}{r}}q^{\invd(f_{\DD}(\sigma,[(A,\epsilon)]))}\\&=& t^{\odesd(\sigma)}s^{\edesd(\sigma)}q^{\invd(\sigma)}\qbinom{n}{r} (1+q^{n-1})(1+q^{n-2})\cdots(1+q^{n-r})
 \end{eqnarray*}
 Summing over all possible $\sigma \in \DD_{n-r}$ finishes the proof.
}

\end{proof}

\begin{lemma}
	Let $X_{\{1,\overline{1}\}}$ be the set of signed permutations in $\DD_n$ such that the descent set is a subset of $\lbrace 1,\overline{1} \rbrace$. 
	Then,
	\begin{equation}
	\sum_{w\in X_{\{1,\overline{1}\}}} t^{\odesd(w)} s^{\edesd(w)} q^{\invd(w)}=t^2\frac{D_n(1,q)}{[n-1]_q!}+t(1-t)\bigg(\frac{2D_n(1,q)}{[n]_q!}-1\bigg)+(1-t).
	\end{equation}
\end{lemma}
\begin{proof}
	Let $Y_1=X_{\{1,\overline{1}\}}$ be the set of signed permutations of $\DD_n$ with the last $n-1$ elements in ascending order and $Y_0$ be the set of signed permutations of $\DD_n$ such that the descent set is a subset of $\lbrace 1 \rbrace$ or $\lbrace \overline{1} \rbrace$. Then, by inclusion-exclusion, we can say that 
	\begin{multline}
	\sum_{w\in X_{\{1,\overline{1}\}}} t^{\odesd(w)} s^{\edesd(w)} q^{\invd(w)}=t^2\bigg(\sum_{w\in Y_1} q^{\invd(w)}\bigg) +t(1-t)\bigg(\sum_{w\in Y_0} q^{\invd(w)}\bigg)+(1-t).
	\end{multline}
	The equation
	\begin{equation}
	\sum_{w\in Y_1} q^{\invd(w)}=\frac{D_n(1,q)}{[n-1]_q!}
	\end{equation}
	comes from \eqref{thm:combin-typed-coeff-new}. 
	
	We just need to show that
	\begin{equation}
	\sum_{w\in Y_0} q^{\invd(w)}=\frac{2D_n(1,q)}{[n]_q!}-1
	\end{equation}
	If we want a descent at $\lbrace 1 \rbrace$ but not at $\lbrace \overline{1} \rbrace$ or vice versa, then we need $|\pi(1)|>|\pi(2)|$. 
	This can be done in the following way. We assign signs to the 
elements of $[n]$ and arrange them in ascending order.   Then, 
choose the sign of the first element accordingly to make it an 
element of $\DD_n$ (i.e. to make the total number of negative signs even).
An element $i$ will either contribute $1$ if it is positive or 
$q^{i-1}$ if it is negative, giving the term $(1+q^{i-1})$. Therefore, 
the total contribution would be $(1+q^0)(1+q)\cdots(1+q^{n-1})$. However, 
this procedure also produces $1,2,\dots,n$ and 
$\overline{1},2,\cdots,n$, out of which we only need the former. 
The latter has a length of $1$ which we subtract to complete 
the proof. 
\end{proof}

\begin{lemma}
	\label{lem:passing_Hni_to_Hniminus1}
	With the above notations, when $i$ is odd, we have
	\begin{eqnarray}
	\sum_{\pi'\in H_{n,i}} t^{\odesd(\pi')}s^{\edesd(\pi')}q^{\invd(\pi')} =&&
	t\frac{D_{i}(s,t,q)D_n(1,q)}{D_{i}(1,q)[n-i]_q!}\nonumber\\
	&+& (1-t) \bigg\{ \sum_{\pi' \in H_{n,i-1}} t^{\odesd(\pi')}
	s^{\edesd(\pi')}q^{\invd(\pi')}\bigg\} \label{eqn:passing_Hni_to_Hniminus1iodd} 
	\end{eqnarray} 
	When $i$ is even, we have
	\begin{eqnarray}
	\sum_{\pi'\in H_{n,i}} t^{\odesd(\pi')}s^{\edesd(\pi')}q^{\invd(\pi')} &=&
	s\frac{D_{i}(s,t,q)D_n(1,q)}{D_{i}(1,q)[n-i]_q!}\nonumber\\&+& (1-s) \bigg\{ \sum_{\pi' \in H_{n,i-1}} t^{\odesd(\pi')}
	s^{\edesd(\pi')}q^{\invd(\pi')}\bigg\}. \label{eqn:passing_Hni_to_Hniminus1ieven} 
	\end{eqnarray} 
\end{lemma}

\begin{proof} 
We at first prove \eqref{eqn:passing_Hni_to_Hniminus1iodd} and therefore 
take $i$ to be odd.  We evaluate 

$ \displaystyle \sum_{(\pi,(A,\epsilon))\in \DD_i\times \sgnbinom{[n]}{n-i}} 
	t^{\odesd(\pi)}s^{\edesd(\pi)}q^{\invd(f_{\DD}(\pi,(A,\epsilon)))}$ 
in a different way as compared to \eqref{thm:cor-typed-coeff}.  
	\begin{eqnarray}
	& & \sum_{(\pi,(A,\epsilon))\in \DD_i\times \sgnbinom{[n]}{n-i}} t^{\odesd(\pi)}s^{\edesd(\pi)}q^{\invd(f_{\DD}(\pi,(A,\epsilon)))} \nonumber  \\
	& = & \sum_{(\pi,(A,\epsilon))\in f_{\DD}^{-1}(H_{n,i})} t^{\odesd(\pi)}s^{\edesd(\pi)}q^{\invd(f_{\DD}(\pi,(A,\epsilon)))} \nonumber \\
	& = & \sum_{(\pi,(A,\epsilon))\in f_{\DD}^{-1}(H_{n,i-1})} t^{\odesd(\pi)}s^{\edesd(\pi)}q^{\invd(f_{\DD}(\pi,(A,\epsilon)))}  \nonumber \\
	& & + \sum_{(\pi,(A,\epsilon))\in f_{\DD}^{-1}(H'_{n,i})} t^{\odesd(\pi)}s^{\edesd(\pi)}q^{\invd(f_{\DD}(\pi,(A,\epsilon)))} \nonumber\\
	& = & \sum_{f_{\DD}(\pi,(A,\epsilon))\in H_{n,i-1}} t^{\odesd(f_{\DD}(\pi,(A,\epsilon)))}s^{\edesd(f_{\DD}(\pi,(A,\epsilon)))}q^{\invd(f_{\DD}(\pi,(A,\epsilon)))}\nonumber\\
	&& \hspace{4 mm} + \frac{1}{t}\bigg\{ \sum_{f_{\DD}(\pi,(A,\epsilon))\in H'_{n,i}} t^{\odesd(f_{\DD}(\pi,(A,\epsilon)))}s^{\edesd(f_{\DD}(\pi,(A,\epsilon)))} q^{\invd(f_{\DD}(\pi,(A,\epsilon)))} \bigg\} \nonumber \\
	& =& \sum_{\pi'\in H_{n,i-1}} t^{\odesd(\pi')}s^{\edesd(\pi')}q^{\invd(\pi')}\nonumber\\&& \hspace{3mm} + \frac{1}{t}\bigg\{ \sum_{\pi'\in H_{n,i}} t^{\odesd(\pi')}s^{\edesd(\pi')}q^{\invd(\pi')} - \sum_{\pi' \in H_{n,i-1}} t^{\odesd(\pi')}s^{\edesd(\pi')}q^{\invd(\pi')} \bigg\} \label{eqn:laststep-typed-rec}.
	\end{eqnarray}
	The second equality follows because $f_{\DD}$ is a bijection between 
	$\DD_i \times \sgnbinom{[n]}{n-i}$ to $H_{n,i}$. For the fourth equality, we have used that $i$ is odd. For the fifth equality, we are again using that $f_{\DD}$ is a bijection and $H'_{n,i}=H_{n,i}-H_{n,i-1}.$ We determine the contribution of each of these three sets. 
	From \eqref{thm:cor-typed-coeff} and \eqref{eqn:laststep-typed-rec}, we have 
	\begin{eqnarray*}
		t\frac{D_{i}(s,t,q)D_n(1,q)}{D_{i}(1,q)[n-i]_q!}&=&(t-1)\Bigg\{ \sum_{\pi' \in H_{n,i-1}} t^{\odesd(\pi')}
		s^{\edesd(\pi')}q^{\invd(\pi')}\Bigg\} \\&&+ \sum_{\pi'\in H_{n,i}} t^{\odesd(\pi')}s^{\edesd(\pi')}q^{\invd(\pi')}
	\end{eqnarray*}
	This completes the proof of \eqref{eqn:passing_Hni_to_Hniminus1iodd}. 
	The proof when $i$ is even is similar and hence is omitted.
\end{proof} 

Our next result is a type D counterpart of the recurrence given
in Theorem 	\ref{thm:typeb_q_rec}.

\begin{theorem}
	\label{thm:typed_q_rec}
\blue{Define $D_2(s,t,q)=(1+tq)^2$.}
When $n \geq 3$, the polynomials $D_n(s,t,q)$ 
satisfy the following recurrence.
	\begin{eqnarray}
	\frac{D_{n}(s,t,q)}{D_{n}(1,q)}& =& \frac{(1-t)^{k+1}(1-s)^{k}}{D_{n}(1,q)}+\frac{2t(1-t)^{k}(1-s)^{k}}{[n]_q!}+\frac{t^2(1-t)^{k-1}(1-s)^{k}}{[n-1]_q!} \nonumber\\
	& & +\sum_{r=0}^{ k-1}  t(1-t)^r(1-s)^{r+1} \frac{D_{n-2r-1}(s,t,q)}{D_{n-2r-1}(1,q)[2r+1]_q!}  \nonumber \\
	& & + \sum_{r=0}^{k-1} s(1-t)^{r}(1-s)^{r}\frac{D_{n-2r}(s,t,q)}{D_{n-2r}(1,q)[2r]_q!} \hspace{1 cm}\mbox{if $n=2k$ is even,}  
	\label{eqn:typed_rec_even} \\
	\frac{D_{n}(s,t,q)}{D_n(1,q)} & =& \frac{(1-t)^{k+2}(1-s)^{k}}{D_n(1,q)} +\frac{2t(1-t)^{k+1}(1-s)^{k}}{[n]_q!}+\frac{t^2(1-t)^{k}(1-s)^{k}}{[n-1]_q!} \nonumber \\ 
	& & +\sum_{r=0}^{k-1} s(1-t)^{r+1}(1-s)^{r} \frac{D_{n-2r-1}(s,t,q)}{D_{n-2r-1}(1,q)[2r+1]_q!} \nonumber \\
	& &+ \sum_{r=0}^{k-1} t(1-t)^{r}(1-s)^r\frac{D_{n-2r}(s,t,q)}{D_{n-2r}(1,q)[2r]_q!} \hspace{1 cm} \mbox{if $n=2k+1$ is odd}. 
	\label{eqn:typed_rec_odd}
	\end{eqnarray}
\end{theorem} 

\begin{proof}
As $H_{n,i}$ is the set of signed permutations in $\DD_n$ 
whose rightmost $(n-i)$ entries form an increasing run, we see 
that $H_{n,n-1}$ must be the whole of $\DD_n$. We first consider 
the case when $n$ is even. By repeatedly applying 
\eqref{eqn:passing_Hni_to_Hniminus1iodd} and 
\eqref{eqn:passing_Hni_to_Hniminus1ieven}, we have 
	\begin{eqnarray*}
		D_{n}(s,t,q)&=& \sum_{\pi \in H_{n,n-1}} t^{\odesd(\pi)}s^{\edesd(\pi)}q^{\invd(\pi)} \\
		&=&t\frac{D_{n-1}(s,t,q)D_n(1,q)}{D_{n-1}(1,q)[1]_q!}+(1-t)\Bigg(\sum_{\pi \in H_{n,n-2}}  t^{\odesd(\pi)}s^{\edesd(\pi)}q^{\invd(\pi)}\Bigg) \\
		&= & t\frac{D_{n-1}(s,t,q)D_n(1,q)}{D_{n-1}(1,q)[1]_q!}+s(1-t)\frac{D_{n-2}(s,t,q)D_n(1,q)}{D_{n-2}(1,q)[2]_q!} \\
		& & + (1-s)(1-t)\sum_{\pi \in H_{n,n-3}} t^{\odesd(\pi)}s^{\edesd(\pi)}q^{\invd(\pi)} \\
		& = & t\frac{D_{n-1}(s,t,q)D_n(1,q)}{D_{n-1}(1,q)[1]_q!}+ s(1-t)\frac{D_{n-2}(s,t,q)D_n(1,q)}{D_{n-2}(1,q)[2]_q!}
		+ \cdots \\
		& & + s(1-t)^{\frac{n}{2}-1}(1-s)^{\frac{n}{2}-2}\frac{D_{2}(s,t,q)D_n(1,q)}{D_{2}(1,q)[n-2]_q!}\\
		&& +(1-t)^{\frac{n}{2}-1}(1-s)^{\frac{n}{2}-1}\Bigg(\sum_{\pi \in X_{\lbrace 1,\overline{1}\rbrace}} t^{\odesd(\pi)}s^{\edesd(\pi)}q^{\invd(\pi)}\Bigg).
	\end{eqnarray*}
	
	This completes the proof of \eqref{eqn:typed_rec_even}. We now consider the case when $n$ is odd. Here, we have
	
	\begin{eqnarray*}
		D_n(s,t,q)&=& \sum_{\pi \in H_{n,n-1}} t^{\odesd(\pi)}s^{\edesd(\pi)}q^{\invd(\pi)} \\
		&=&s\frac{D_{n-1}(s,t,q)D_n(1,q)}{D_{n-1}(1,q)[1]_q!} 
		+(1-s)\Bigg(\sum_{\pi \in H_{n,n-2}}  t^{\odesd(\pi)}s^{\edesd(\pi)}q^{\invd(\pi)}\Bigg) \\
		&=&s\frac{D_{n-1}(s,t,q)D_n(1,q)}{D_{n-1}(1,q)[1]_q!}\\
		& & +(1-s)\left (t\frac{D_{n-2}(s,t,q)D_n(1,q)}{D_{n-2}(1,q)[2]_q!}+(1-t)\Big(\sum_{\pi \in H_{n,n-3}} t^{\odesd(\pi)}s^{\edesd(\pi)}q^{\invd(\pi)}\Big) \hspace{-1.4 mm} \right )
	\end{eqnarray*}	
and we can continue as in the case for $n$ being even, 
to complete the proof of \eqref{eqn:typed_rec_odd}. This 
completes the proof. 	
\end{proof}

\subsection{Type D generating functions}
\label{subsec:type_d_gen_fns}
We again cast the recurrences in egf language to get generating functions.
We begin with our proof of Theorem 
\ref{thm:carlitz_scoville_typed_eulerian_q-anal}.

\begin{proof}[Proof of Theorem \ref{thm:carlitz_scoville_typed_eulerian_q-anal}]
Recurrences \eqref{eqn:typed_rec_even} and \eqref{eqn:typed_rec_odd} 
give rise to this following.
	\begin{eqnarray}
	\label{eqn:thm9eqn1} 
	& &  
	\hspace{-5 mm} \mathcal{D}_0\bigg(1-s\cosh_q(Mu)-\frac{s(1-t)\sinh_q(Mu)}{M}\bigg)+\mathcal{D}_1\bigg(1-t\cosh_q(Mu)-\frac{t(1-s)\sinh_q(Mu)}{M}\bigg)\nonumber\\
	& =& \mathrm{OD} + \mathrm{ED}.
	\end{eqnarray}
	\noindent 
	Changing $u$ to $-u$ gives us
	\begin{eqnarray}
	& &   \hspace{-5 mm}
	\mathcal{D}_0\bigg(1-s\cosh_q(Mu)+\frac{s(1-t)\sinh_q(Mu)}{M}\bigg)-\mathcal{D}_1\bigg(1-t\cosh_q(Mu)+\frac{t(1-s)\sinh_q(Mu)}{M}\bigg)\nonumber\\
	& = & -\mathrm{OD} + \mathrm{ED}
	\label{eqn:thm9eqn2}.
	\end{eqnarray}
	
	
Solving the above two equations for $\cD_0$ and $\cD_1$ completes the 
proof.
\end{proof} 

We can now prove Theorem 
\ref{thm:type-d-alt-eul}.

\blue{
\begin{proof} (Of Theorem \ref{thm:type-d-alt-eul})
As done in the proof of Theorem 
\ref{thm:from_carlitz_scoville_typeb_alt_eulerian_q-anal},
one can check when $n=2k$ is even, that
	$\displaystyle \hat{D}_{2k}(s,t,q)  = 
	s^{k} D_{2k}(1/s,t,q)$ and when $n=2k+1$  is odd, that
$\hat{D}_{2k+1}(s,t,q)
	=  s^{k+1}D_{2k+1}\bigg(\frac{1}{s},t,q \bigg).$  
The other details follow as in the proof of Theorem 
\ref{thm:from_carlitz_scoville_typeb_alt_eulerian_q-anal},
completing the proof.
\end{proof}
}

Using Theorem \ref{thm:carlitz_scoville_typed_eulerian_q-anal}
we get a type D counterpart of Theorem  
\ref{thm:carlitz_scoville}. Define 
$$D_n(s_0,s_1,t_0,t_1,q)=\sum_{w\in \DD_n} 
s_0^{\eascd(w)}s_1^{\oascd(w)}t_0^{\edesd(w)}
t_1^{\odesd(w)}q^{\invd(w)}.$$ 
Further, define the generating functions 
\begin{eqnarray}
\label{eqn:defn_q_egf-type-d-five-var-even}
\cD_0(s_0,s_1,t_0,t_1,q,u) & = & \sum_{k\geq 1}D_{2k}(s_0,s_1,t_0,t_1,q)\frac{u^{2k}}{D_{2k}(1,q)},\\
\label{eqn:defn_q_egf-type-d-five-var-odd}
\cD_1(s_0,s_1,t_0,t_1,q,u) & = & \sum_{k\geq 1}D_{2k+1}(s_0,s_1,t_0,t_1,q)\frac{u^{2k+1}}{D_{2k+1}(1,q)}.
\end{eqnarray}

\magenta{
We move to our type D counterpart of Theorem
\ref{thm:carlitz_scoville}.  
Recall $\cD_0(s,t,q,u)$ and $\cD_1(s,t,q,u)$ from 
\eqref{eqn:defn_four_var_egf_type-D}.   
\begin{theorem}
\label{thm:pan_zeng_four_var_type-d-version}
We have the egf
\begin{eqnarray}
\cD_0(s_0,s_1,t_0,t_1,q,u) = 
\frac{ \frac{1}{s_0}[T(\mathrm{ED})(s_1-t_1\cosh_q(mu))]+T(\mathrm{OD})
(\frac{t_1(s_0-t_0)\sqrt{s_0s_1}}{ms_0}\sinh_q(mu))}
{s_0s_1-(s_0t_1+s_1t_0)\cosh_q(mu)+t_0t_1\mathrm{e}_q(mu)\mathrm{e}_q(-mu)},\\ 
\cD_1(s_0,s_1,t_0,t_1,q,u)=\frac{\frac{s_1\sqrt{s_1}T(\mathrm{OD})}{\sqrt{s_0}}(s_0-t_0\cosh_q(mu))+T(\mathrm{ED})(\frac{s_1t_0(s_1-t_1)}{m}\sinh_q(mu))}{s_0s_1-(s_0t_1+s_1t_0)\cosh_q(mu)+t_0t_1\mathrm{e}_q(mu)\mathrm{e}_q(-mu)}. 
\end{eqnarray}
where 
\begin{eqnarray*}
T(\mathrm{OD}) & =& \frac{\sqrt{s_0s_1}ut_1^2}{s_1^2}(\cosh_q(mu)-1)+
\frac{(s_1-t_1)m}{(s_0-t_0)\sqrt{s_0s_1}}(\sinh_{D}(mu;q)-mu) \\
 & & +\frac{2t_1(s_1-t_1)\sqrt{s_0s_1}}{s_1^2m}(\sinh_q(mu)-mu), \\
T(\mathrm{ED}) & = & \frac{2t_1}{s_1}(\cosh_q(mu)-1)+\frac{ut_1^2(s_0-t_0)
\sqrt{s_0s_1}}{s_1^2s_0m}\sinh_q(mu)+\frac{(s_1-t_1)}{t_1}(\cosh_{D}(mu;q)-1).
\end{eqnarray*}
\end{theorem}
\begin{proof}
We proceed as we did in the proof of Theorem 
\ref{q-egf-type-b-counterpart-pan-zeng}.
It is easy to see that
\begin{eqnarray*}
\cD_0(s_0,s_1,t_0,t_1,q,u) & = &\frac{s_1}{s_0}\mathcal{D}_0(\frac{t_0}{s_0},\frac{t_1}{s_1},q,\sqrt{s_0s_1}u), \\
\cD_1(s_0, s_1, t_0, t_1, q,u) & = & \frac{\sqrt{s_1}}{\sqrt{s_0}}\mathcal{D}_1(\frac{t_0}{s_0},\frac{t_1}{s_1},q,\sqrt{s_0s_1}u).
\end{eqnarray*}
We denote by $T$ the transformation that sends $s$ to 
$\displaystyle \frac{t_0}{s_0}$, $t$ to $\displaystyle 
\frac{t_1}{s_1}$ and $u$ to $\sqrt{s_0s_1}u$.  It is easy to see that  
$T(OD)$ and $T(ED)$ are as given above, completing the proof.
\end{proof}
}

\subsection{Type D $q$-analogue of Hyatt's recurrences}
\label{subsec:hyatt-type-d}

We give our $q$-analogue of Hyatt-type recurrences in 
this subsection.

\begin{theorem}
	\label{thm:typed_plus}
	For even n,
	\begin{eqnarray*}
		\label{eqn:typed_plus_biv_even}
		\sum_{w \in \DD_n^+} t^{\odesd(w)} s^{\edesd(w)}q^{\invd(w)}=
		\sum_{r=0}^{\frac{n}{2}-1}
		q^{\binom{2r+1}{2}}\qbinom{n}{2r+1}D_{n-2r-1}(s,t,q)(s-1)^r (t-1)^r \\
		+ \sum_{r=1}^{\frac{n}{2}}
		q^{\binom{2r}{2}}\qbinom{n}{2r}D_{n-2r}(s,t,q)(s-1)^{r-1} (t-1)^r.
	\end{eqnarray*}
	For odd n,
	\begin{eqnarray*}
		\label{eqn:typed_plus_biv_odd}
		\sum_{w \in \DD_n^+} t^{\odesd(w)} s^{\edesd(w)}q^{\invd(w)}=
		\sum_{r=0}^{\lfloor \frac{n}{2}\rfloor}
		q^{\binom{2r+1}{2}}\qbinom{n}{2r+1}D_{n-2r-1}(s,t,q)(s-1)^r (t-1)^r\\ +
		\sum_{r=1}^{\lfloor \frac{n}{2} \rfloor}
		q^{\binom{2r}{2}}\qbinom{n}{2r}D_{n-2r}(s,t,q)(s-1)^{r} (t-1)^{r-1}.
	\end{eqnarray*}
\end{theorem}
\begin{proof}
	Define $_\DD\hat{A}_k$ to be the signed permutations in $\DD_n^+$
	that have their rightmost $k+1$ elements being positive and 
arranged in descending order. Thus, the first
	$n-k-1$ elements must have an even number of negative signs. With this
	observation note that the map $f'':\DD_{k} \times \binom{[n]}{n-k}
	\rightarrow$ $_{\DD}\hat{A}_{n-k-1}$ that carries $(\psi,A)$ to the signed
	permutation $\psi_{[n]-A}a_1a_2\cdots a_{n-k}$ ($a_1>a_2>\cdots
	a_{n-k}>0$),  that is,
$$f''(\psi,A)=\psi_{[n]-A},a_1,a_2,\cdots,a_{n-k}.$$
	is a bijection from $\DD_{k}\times \binom{[n]}{n-k}$ onto
	$_\DD\hat{A}_{n-k-1}$.
	
Write ${_\DD A}_k(s,t,q)=\sum_{w\in _\DD \hat{A}_k} 
t^{\odesd(w)} s^{\edesd(w)} q^{\invd(w)}$.  We will abbreviate the 
LHS as ${_\DD A}_k$ for brevity.
	The following recurrences are then easy to prove.
	For even $n$, we have
	\begin{eqnarray}
	\label{eqn:typed_even_even}
	q^{\binom{2r}{2}}\qbinom{n}{2r} D_{n-2r}(s,t,q)s^rt^r & = & s{_\DD
		A}_{2r-1}-(s-1){_\DD A}_{2r}. \\
	\label{eqn:typed_even_odd}
	q^{\binom{2r+1}{2}}\qbinom{n}{2r+1} D_{n-2r-1}(s,t)s^rt^{r+1} & = &
	t {_\DD A}_{2r}-(t-1){_\DD A}_{2r+1}.
	\end{eqnarray}
	For odd $n$, we have
	\begin{eqnarray}
	\label{eqn:typed_odd_even}
	q^{\binom{2r}{2}}\qbinom{n}{2r} D_{n-2r}(s,t)s^{r}t^r & = & t{_\DD
		A}_{2r-1}-(t-1){_\DD A}_{2r}. \\
	\label{eqn:typed_odd_odd}
	q^{\binom{2r+1}{2}}\qbinom{n}{2r+1}
	D_{n-2r-1}(s,t)s^{r+1}t^{r} & = & s{_\DD A}_{2r}-(s-1){_\DD A}_{2r+1}.
	\end{eqnarray}
	The proofs of these recurrences are along the same lines as the proofs of \eqref{eqn:typeb_even_even},\eqref{eqn:typeb_even_odd},\eqref{eqn:typeb_odd_even} and \eqref{eqn:typeb_odd_odd}. The only ambiguity might be when $r=\nhalf-1$, but this is easily
	resolved as in $_\DD \hat{A}_{2r}$ the rightmost $n-1$ elements 
	are positive and descending for even $n$ or $_\DD \hat{A}_{2r+1}$ when
	$n$ is odd, the first element has to be positive due to the constraint
	that there are an even number of negative signs. Therefore, there is no possibility of $w_1+w_2$ being lesser than $0$.
\end{proof}

To preserve elements being in $\DD_n$, we consider the map that
flips the sign of all elements when $n$ is even and the map
that flips the sign of all elements except the first when 
$n$ is odd.

\begin{lemma}
	\label{lem:sgnflipmaptyped}
	Let $f_D: \DD_n \rightarrow \DD_n$ be the involution that sends
	$w=w_1,\cdots,w_n$ to $\overline{w}=\overline{w_1},\dots,
	\overline{w_n}$ if $n$ is even and $w=w_1,\cdots,w_n$ to
	$\overline{w}=w_1,\overline{w_2}\dots,
	\overline{w_n}$ if $n$ is odd. Then, we have the following:
	\begin{enumerate}
		\item When $n=2k+1$, we have $\odesd(w)+\odesd(f_D(w))=k+1$ and
		when $n=2k$, we have $\odesd(w)+\odesd(f_D(w))=k+1$.
		\item When $n=2k+1$, we have $\edesd(w)+\edesd(f_D(w))=k$ and when
		$n=2k$, we have $\edesd(w)+\edesd(f_D(w))=k-1$.
		\item $\invd(w)+\invd(f_D(w))=n(n-1)$.
	\end{enumerate}
\end{lemma}
\begin{proof} The proof of the first two assertions are straightforward
	and hence omitted. For the third part,
%
	recall that  $\invd(w)=\inv_B(w)-|\Negs(w)|$.
	Thus, we have, for even $n$,
	\begin{eqnarray*}
		\invd(w)+\invd(\overline{w})    
		&=&\inv_B(w)+\inv_B(\overline{w})-|\Negs(w)|-|\Negs(\overline{w})|\\
		& = &  n^2-n=n(n-1).
	\end{eqnarray*}
	When $n$ is odd, it is easy to see that
	$\invd({w_1,\overline{w_2},\ldots,\overline{w_n}}) =
	\invd({\overline{w_1},\ldots,\overline{w_n}})$.
	The rest follows from the previous argument.
	The proof is complete.
\end{proof}

\subsection{Symmetry results}
\label{subsec:symm_type-d}

In this Subsection, we give our type D counterparts of 
our symmetry results.


\begin{theorem}
	\label{thm:typed_minus}
	We have  
	$$
	D_n^-(s,t,q)  =  
	\left\{
	\begin{array}{ll}
	q^{n(n-1)}s^{k}t^{k+1} D_n^+(s^{-1},t^{-1},q^{-1}) & \mbox{when $n = 2k+1$,} \\
	q^{n(n-1)}s^{k-1}t^{k+1} D_n^+(s^{-1},t^{-1},q^{-1}) & \mbox{when $n=2k$}.
	\end{array} \right.
	$$
	
	Therefore, we have
	$$
	D_n(s,t,q) =
	\left\{
	\begin{array}{ll}
	D_n^+(s,t,q)+q^{n(n-1)}s^{k}t^{k+1} D_n^+(s^{-1},t^{-1},q^{-1})
	& \mbox{when $n=2k+1$}, \\
	D_n^+(s,t,q)+q^{n(n-1)}s^{k-1}t^{k+1} D_n^+(s^{-1},t^{-1},q^{-1})
	& \mbox{when $n=2k$.}\\
	\end{array}
	\right.
	$$
\end{theorem}
\begin{proof}
	Let $f_D:\DD_n^+ \rightarrow \DD_n^-$ be the map described earlier.
	By Lemma \ref{lem:sgnflipmaptyped}, when $n=2k$, we have
	\begin{eqnarray*}
		\sum_{w\in \DD_n^-} t^{\odesd(w)}s^{\edesd(w)}q^{\invd(w)}&=&\sum_{w\in \DD_n^+} t^{\odesd(f_D(w))}s^{\edesd(f_D(w))}q^{\invd(f_D(w))}\\
		&=&\sum_{w\in \DD_n^+}t^{k+1-\odesd(w)}s^{k-1-\edesd(w)}q^{n(n-1)-\invd(w)}\\
		&=&q^{n(n-1)}s^{k-1}t^{k+1}\sum_{w\in \DD_n^+} t^{-\odesd(w)}s^{-\edesd(w)}q^{-\invd(w)}.
	\end{eqnarray*}
	When $n=2k+1$, we have
	\begin{eqnarray*}
		\sum_{w\in \DD_n^-} t^{\odesd(w)}s^{\edesd(w)}q^{\invd(w)}&=&\sum_{w\in \DD_n^+} t^{\odesd(f_D(w))}s^{\edesd(f_D(w))}q^{\invd(f_D(w))}\\
		&=&\sum_{w\in \DD_n^+} t^{k+1-\odesd(w)}s^{k-\edesd(w)}q^{n(n-1)-\invd(w)}\\
		&=&q^{n(n-1)}s^{k}t^{k+1}\sum_{w\in \DD_n^+} t^{-\odesd(w)}s^{-\edesd(w)}q^{-\invd(w)}.
	\end{eqnarray*}
	completing the proof.
\end{proof}

Since the following corollary is straightforward, we only
state it and omit its proof.

\begin{corollary}[Type-D Symmetry]
	\label{thm: type_d_reciprocal}
	We have
	$$
	D_{n}(s,t,q)=
	\left\{
	\begin{array}{ll}
	q^{n(n-1)}s^{k}t^{k+1} D_{n}(s^{-1},t^{-1},q^{-1})
	& \mbox{when $n=2k+1$}, \\
	q^{n(n-1)}s^{k-1}t^{k+1} D_{n}
	(s^{-1},t^{-1},q^{-1})
	& \mbox{when $n=2k$.}\\
	\end{array}
	\right.
	$$
\end{corollary}

\section{Snakes}
\label{sec:snakes}
A $\mathrm{snake}$ in $\BB_n$ is a signed permutation $w \in \BB_n$ satisfying
$0<w_1>w_2<\cdots$. Let $\Snake_n^B$ be the set of $\mathrm{snakes}$ in
$\BB_n$ and denote $|\Snake_n^B|$ by $S^B_n$.  
The paper by Arnol'd \cite{arnold-calculus-snakes} is a good reference for
this topic.  Let 
$S^B_n(q)=\sum\limits_{w\in \Snake_n^B} q^{\invb(w)}$.
Springer in \cite{springer-remarks-combin}
showed the following.
\begin{theorem}[Springer]
\label{thm:snake-gf}
The following is the egf for the numbers $S^B_n$:
    $$\sum\limits_{n\ge 0}S^B_n\frac{u^n}{n!}=\frac{1}{\cos(u)-\sin(u)}.$$
\end{theorem}

The following corollary of Theorem \ref{q-egf-type-b-counterpart-pan-zeng}
is now easy.
\begin{corollary}
\label{thm:q-snake-gf}
We have the following egf of the $S^B_n(q)$ polynomials:
    $$\sum\limits_{n\ge 0} S^B_n(q)
\frac{u^n}{B_n(1,q)}=
\frac{\cos_q(u)\cos_B(u;q)+(\sin_q(u)-1)\sin_B(u;q)}{\cos_q(u)}$$
\end{corollary}
\begin{proof}
Setting $s_1=t_0=0$ and $t_1=s_0=1$ in both $H_0(s_0,s_1,t_0,t_1,q,u)$ 
and $H_1(s_0,s_1,t_0,t_1,q,u)$
from Theorem \ref{q-egf-type-b-counterpart-pan-zeng} and adding
completes the proof.
\end{proof}

It is easy to see that setting $q=1$ in Corollary \ref{thm:q-snake-gf} 
gives us Theorem \ref{thm:snake-gf}.

\subsection{D-snakes}
A $\mathrm{d}$-$\mathrm{snake}$ in $\DD_n$ is a signed permutation 
$w$ in $\DD_n$ that satisfies $-w_2>w_1>w_2<w_3>\dots w_n$. 
Let $\mathrm{Snake}^D_n$ be the set of all 
$\mathrm{d}$-$\mathrm{snakes}$ in $\DD_n$.  Denote 
$|\mathrm{Snake}^D_n|$ by $S^D_n$.
Let $S^D_n(q)=\sum_{w \in \mathrm{Snake}^D_n} q^{\invd(w)}.$  Define
$\SD_0(q,u) = \sum\limits_{n\ge 1} S^D_{2n}(q) \frac{u^{2n}}{D_{2n}(1,q)}$ and
$\SD_1(q,u) = \sum\limits_{n\ge 1} S^D_{2n+1}(q) \frac{u^{2n+1}}{D_{2n+1}(1,q)}$.

\begin{corollary}
\label{thm:q-dsnake-gf}
We have the following egf of the $S^D_n(q)$ polynomials:
    \begin{eqnarray}
	\label{eqn:snakes-d-q-analogue-even}
    \SD_0(q,u) &=& \frac{-2\cos^2_q(u)+\cos_q(u)(\cos_D(u;q)-1)-2\sin^2_q(u)+\sin_q(u)\sin_D(u;q)}{-\cos_q(u)},\\
	\label{eqn:snakes-d-q-analogue-odd}
    \SD_1(q,u) & = &
\frac{-2\sin_q(u)+u\cos_q(u)+\sin_D(u;q)}{-\cos_q(u)}.
    \end{eqnarray}
\end{corollary}
\begin{proof}
Set $t=1/t$, $u=u \sqrt{t}$, multiply by $t$ and setting $s=t=0$ in 
Theorem \ref{thm:carlitz_scoville_typed_eulerian_q-anal} gives us 
	\eqref{eqn:snakes-d-q-analogue-even}.
Set $t=1/t$, $u=u \sqrt{t}$, multiply by $\sqrt{t}$ and setting $s=t=0$ in 
Theorem \ref{thm:carlitz_scoville_typed_eulerian_q-anal} gives us 
	\eqref{eqn:snakes-d-q-analogue-odd}.
\end{proof}

Setting $q=1$ in Corollary  \ref{thm:q-dsnake-gf}
gives us the following egf which is given by Springer.
\begin{corollary}[Springer]
    The egf for the $S^D_n$ is:
    \begin{eqnarray*}
   \sum\limits_{n\ge 1} S^D_{2n} \frac{u^{2n}}{(2n)!} & = & \frac{\cos(u)-\cos(2u)-1}{-\cos(2u)},\\
   \sum\limits_{n\ge 1} S^D_{2n+1} \frac{u^{2n+1}}{(2n+1)!} & = &\frac{-\sin(2u)+u\cos(2u)+\sin(u)}{-\cos(2u)}.
    \end{eqnarray*}
\end{corollary}



\section*{Acknowledgement} 
The first author acknowledges SERB-National Post Doctoral Fellowship 
(File No. PDF/2021/001899) during the preparation of this work and 
profusely thanks Science and Engineering Research Board, Govt. of India 
for this funding. The first author also acknowledges excellent 
working conditions in the Department of Mathematics, Indian 
Institute of Science. The second author thanks National Board of 
Higher Mathematics for funding. The second author  also thanks 
Indian Institute of Technology, Bombay for its excellent 
working conditions.

\bibliographystyle{acm}
\bibliography{main}
\end{document}